\documentclass[a4paper,10pt,oneside]{amsart}

\theoremstyle{plain}
\newtheorem{lemma}{Lemma}[section]
\newtheorem{proposition}[lemma]{\textbf{Proposition}}
\newtheorem{theorem}[lemma]{\textbf{Theorem}}
\newtheorem*{introtheorem}{\textbf{Theorem}}
\newtheorem{cor}[lemma]{\textbf{Corollary}}
\theoremstyle{definition}
\newtheorem{definition}[lemma]{\textbf{Definition}}
\newtheorem*{tmpdefinition}{\textbf{Temporary Definition}}

\theoremstyle{remark}
\newtheoremstyle{indenteddefinition}{8pt}{8pt}{\addtolength{\leftskip}{
1.4em}}{-1.4em}{\itshape}{.}{ }{}
\theoremstyle{indenteddefinition}
\newtheorem{remark}[lemma]{Remark}

\newcommand{\R}{\mathbb{R}}
\newcommand{\C}{\mathbb{C}}

\newcommand{\p}{\mathbb{P}}

\renewcommand\epsilon{\varepsilon}
\renewcommand\tilde{\widetilde}
\newcommand{\tth}{\thinspace}

\usepackage[english]{babel}
\usepackage[T1]{fontenc}
\usepackage{amsmath,amssymb,amsthm,amsfonts}
\usepackage{latexsym}
\usepackage{mathrsfs}
\usepackage{faktor}
\usepackage{nicefrac}
\usepackage[all]{xy}
\usepackage{changepage}
\usepackage{multirow}
\usepackage{enumerate}
\usepackage{overpic}
\usepackage{graphics}
\usepackage{graphicx}
\usepackage{subfigure}
\usepackage{todonotes}
\usepackage{algorithm}
\usepackage{algpseudocode}

\algrenewcommand\algorithmicwhile{\textbf{While}}
\algrenewcommand\algorithmicfor{\textbf{For}}
\algrenewcommand\algorithmicdo{\textbf{Do}}
\algrenewcommand\algorithmicif{\textbf{If}}
\algrenewcommand\algorithmicthen{\textbf{Then}}
\algrenewcommand\algorithmicelse{\textbf{Else}}
\algrenewcommand\algorithmicend{\textbf{End}}
\algrenewcommand\algorithmicreturn{\textbf{Return}}

\begin{document}

%
%
%
%
%
%
%
%
%

\title{M\"obius Photogrammetry}

\author[Gallet]{Matteo Gallet}

\address{
Research Institute for Symbolic Computation \\
Johannes Kepler University \\
Altenberger Strasse 69 \\
4040 Linz, Austria.}

\email{mgallet@risc.jku.at}

\author[Nawratil]{Georg Nawratil}

\address{
Institute of Discrete Mathematics and Geometry\\  
Vienna University of Technology \\
Wiedner Hauptstrasse 8-10/104 \\
1040 Vienna, Austria.}

\email{nawratil@geometrie.tuwien.ac.at}

\author[Schicho]{Josef Schicho}

\address{
Research Institute for Symbolic Computation \\
Johannes Kepler University \\
Altenberger Strasse 69 \\
4040 Linz, Austria.}

\email{josef.schicho@risc.jku.at}



\begin{abstract} 
Motivated by results on the mobility of mechanical devices called pentapods,
this paper deals with a mathematically freestanding problem, which we call
\emph{M\"obius Photogrammetry}. Unlike traditional photogrammetry, which tries
to recover a set of points in three--dimensional space from a finite set of
central projection, we consider the problem of reconstructing a vector of points
in $\R^3$ starting from its orthogonal parallel projections. Moreover,
we assume that we have partial information about these projections, namely that
we know them only up to M\"obius transformations. The goal in this case is to
understand to what extent we can reconstruct the starting set of points, and to
prove that the result can be achieved if we allow some uncertainties in the
answer. Eventually, the techniques developed in the paper allow us to show that
for a pentapod with mobility at least two, either some anchor points are
collinear, or platform and base are similar, or they are planar and affine
equivalent.
\end{abstract}

\maketitle


\section{Introduction}
\label{photogrammetry:introduction}

In this paper we consider the following problem: given a vector of~$5$ points
$\vec{A} = (A_1, \ldots, A_5)$ in $\R^3$, suppose we have partial information
about its orthogonal projections along all directions in~$\R^3$; in particular,
we suppose to know each of them only up to M\"obius transformations of the
plane. Then we ask if and to what extent we can extract information on~$\vec{A}$ 
starting from this partial knowledge about its orthogonal projections. 

In order to deal with this question, in Section~\ref{photogrammetry:setup} we
start formalizing it in the following way: 
\begin{itemize}
	\item[$\cdot$]
		the set of directions of~$\R^3$ is identified with the unit sphere~$S^2$;
	\item[$\cdot$]
		a set of~$5$ points in the plane, considered up to M\"obius
transformations, gives a point in the moduli space~$M_5$ of five points 
in~$\p^1_{\C}$.
\end{itemize}
So all the information about orthogonal projections of a vector~$\vec{A}$ of
points can be encoded in one function $f_{\vec{A}}: S^2 \longrightarrow M_5$,
which we call a \emph{M\"obius camera}. In Section~\ref{photogrammetry:camera}
we give a formal definition of the map~$f_{\vec{A}}$ and we explain how it can
be thought as a map between projective varieties. We explore some properties of
the M\"obius camera, in particular we relate the degree of the image 
of~$f_{\vec{A}}$ (which is always a curve if the points are not all aligned) 
with the geometric configuration of the points $\{ A_i \}$. The main result of 
this section is Theorem~\ref{theorem:photo}:

\begin{introtheorem}
	Let $\vec{A}$ and $\vec{B}$ be two $5$-tuples of points in~$\R^3$ such that
no~$4$ of them are collinear. Assume that $f_{\vec{A}}(S^2)$ and
$f_{\vec{B}}(S^2)$ are equal as curves in~$M_5$.
	If $\vec{A}$ is coplanar, then~$\vec{B}$ is also coplanar and affine
equivalent to~$\vec{A}$.
	If $\vec{A}$ is not coplanar, then $\vec{B}$ is similar to~$\vec{A}$.
\end{introtheorem}

Eventually, Section~\ref{photogrammetry:pentapods} presents an application of
the theory developed so far to pentapods with mobility greater than or equal
to~$2$. In fact, and this is where the authors took one of the motivations for
this paper, Theorem~3.19 of~\cite{GalletNawratilSchicho} gives a necessary
condition for mobility~2 of $n$-pods in the form of a disjunction of several
statements, one of which being: ``\emph{there are infinitely many pairs $(L,R)$
of elements of~$S^2$ such that the points $\pi_L(p_1), \ldots, \pi_L(p_n)$ and
$\pi_R(P_1), \ldots, \pi_R(P_n)$ differ by an inversion or a similarity}''.
We focus on this case, and using Theorem~\ref{theorem:photo} we show
that base and platform points are either similar or planar and affine
equivalent.

\section{Setting up the M\"obius photogrammetry problem}
\label{photogrammetry:setup}

We are going to consider the following problem: given a vector $\vec{A} = (A_1, 
\ldots, A_5)$ of~$5$ points in~$\R^3$, we want to define a map~$f_{\vec{A}}$, 
which we will call \emph{M\"obius camera}, associating to each direction 
$\varepsilon \in S^2$ the orthogonal projection of~$\vec{A}$
along~$\varepsilon$, considered up to M\"obius transformations. Moreover,
starting from the collection of all orthogonal projections of~$\vec{A}$, we want
to understand to what extent we can reconstruct~$\vec{A}$.

In order to set up this photogrammetric problem in a formal way, first of all 
we have to make clear what do we mean by ``consider up to M\"obius
transformations''. Recall that a \emph{M\"obius transformation} is a 
map~$g$ of the complex projective line~$\p^1_{\C}$ to itself of the form
\[
	g: \left\{ \begin{array}{rcl}
		(1 : z) & \mapsto & \left( 1 : \frac{a \tth z + b}{c \tth z + d} \right)
\quad \mathrm{if\ } z \neq -\nicefrac{d}{c} \\
		(1 : -\nicefrac{d}{c}) & \mapsto & (0 : 1) \\
		(0 : 1) & \mapsto & \left(1 : \nicefrac{a}{c} \right)
	\end{array} \right. \quad \quad \textrm{where } ad - bc \neq 0 
\]
with the convention that $g(0:1) = (0 : 1)$ if $c = 0$. If we are given
two $n$-tuples $(m_1, \ldots, m_n)$ and $(n_1, \ldots, n_n)$ of points in the
plane~$\R^2$, we say that they are \emph{M\"obius equivalent} if, once we
identify~$\R^2$ with $\C \hookrightarrow \p^1_{\C}$, there is a M\"obius
transformation~$g$ sending $m_i$ to $n_i$ for every $i \in \{1,\ldots, n\}$.

The rest of this and the next section are aimed to define the
concept of \emph{M\"obius camera}. We will clarify what are the domain and the 
codomain of this map, and what is its explicit formulation.

\smallskip
We start discussing the domain of our desired map: as described before, it
should be~$S^2$, thought as the set of directions in $\R^3$. On the other hand,
we would like it to be an algebraic variety. We are going to see now that not
only $S^2$ can be considered as a complex projective curve, but it also
naturally carries the structure of a real variety. This property will be crucial
in the proofs of Subsection \ref{photogrammetry:camera:properties}.

\begin{definition}
	A \emph{real structure} on a complex variety is a pair $(X,\alpha)$, where~$X$
is a complex variety and $\alpha$ is an anti-holomorphic involution (see
\cite{Silhol}, Chapter~1, Proposition~1.3).
\end{definition}
\begin{remark}
\label{remark:real_structure_S2}
	An example of a real structure is given by the complex projective 
space~$\p^n_{\C}$ together with componentwise complex conjugation. One can prove 
that there exist exactly two real structures on~$\p^1_{\C}$ (up to 
isomorphism), and they are given by the following two involutions:
	\[ (s,t) \; \mapsto \; (\overline{s},\overline{t}) \quad \mathrm{and} \quad
(s:t) \; \mapsto \; (-\overline{t}:\overline{s}) \]
	In particular, the fixed points of the first involution are precisely the
closed points of~$\p^1_{\R}$, while the second one does not have any fixed
point. Moreover there is a natural bijection between $\p^1_{\C}$ and $S^2$, and 
under this bijection the second involution corresponds to the antipodal map. 
Hence we can think of~$S^2$ as a real algebraic variety, whose 
anti-holomorphic involution is given by the antipodal map. 
\end{remark}

The following result provides another identification of~$S^2$ with an algebraic
curve which is simply a different projective embedding of~$\p^1_{\C}$, but
which enables us to perform the computations needed to define the M\"obius
camera in a simpler way.

\begin{lemma}
\label{lemma:bijection_S2_conic}
	There is a bijection $\gamma: S^2 \longrightarrow C = \big\{ x^2 + y^2 + z^2 = 0 \big\} \subseteq \p^2_{\C}$ such that the following diagram commutes:
	\[ \xymatrix@C=3cm{S^2 \ar[r]^{\mathrm{antipodal\ map}} \ar[d]^{\cong}_{\gamma} & S^2 \ar[d]_{\cong}^{\gamma} \\ C \ar[r]^{\mathrm{componentwise}}_{\mathrm{conjugation}} & C} \]
\end{lemma}
\begin{proof}
	Let $\varepsilon \in S^2$, then pick $\varepsilon', \varepsilon''$ in the orthogonal space $\left\langle \varepsilon \right\rangle^{\perp}$ such that
	\begin{itemize}
		\item[$\cdot$]
			$\varepsilon', \varepsilon''$ form an orthonormal basis of $\left\langle \varepsilon \right\rangle^{\perp}$;
		\item[$\cdot$]
			$\varepsilon, \varepsilon', \varepsilon''$ form a right basis of $\R^3$, namely $\det{\left(\begin{smallmatrix} \varepsilon & \varepsilon' & \varepsilon'' \end{smallmatrix}\right)}  > 0$.
	\end{itemize}
	If $\varepsilon' = (\lambda', \mu', \nu')$ and $\varepsilon'' = (\lambda'', \mu'', \nu'')$, then we consider the vector
	\[ \varepsilon' + i \tth \varepsilon'' \; = \; (\lambda' + i \tth \lambda'', \mu' + i \tth \mu'', \nu' + i \tth \nu'') \in \C^3 \]
	By a direct computation one can check that the point in~$\p^2_{\C}$ given by
$\varepsilon' + i \tth \varepsilon''$ lies on~$C$. We notice that a different
choice of $\varepsilon'$ and $\varepsilon''$ leads to the same point 
in~$\p^2_{\C}$. Then the map $\gamma: \varepsilon \mapsto \varepsilon' + i \tth
\varepsilon''$ is well-defined and satisfies the requirements of the thesis.
\end{proof}
\begin{remark}
\label{remark:three_descriptions_S2}
	Recalling Remark~\ref{remark:real_structure_S2}, we have that $S^2$ is in 
bijection with~$\p^1_{\C}$, and via the previous map~$\gamma$ we obtain an 
isomorphism of real algebraic curves between~$\p^1_{\C}$ and~$C$ given by 
homogeneous polynomials of degree~$2$. We get the following triangle of 
bijections and involutions:
	\[
		\hspace{2.5em} \resizebox{0.9\textwidth}{!}{\xymatrix{ \Big( S^2, \textrm{antipodal map} \Big) \ar@{<->}[rd] \ar@{<->}[rr] & & \Big( \big\{ x^2 + y^2 + z^2 = 0 \big\}, \textrm{componentwise conj.} \Big)  \ar@{<->}[ld] \\
		& \Big( \p^1_{\C}, (s:t) \mapsto (-\overline{t}: \overline{s}) \Big) & }}
	\]
\end{remark}

The identification of~$S^2$ with the conic $C \subseteq \p^2_{\C}$ becomes very
useful when we want to deal with orthogonal projections. 

\begin{definition}
\label{definition:orthogonal_projection}
	Given a unit vector $\epsilon \in S^2$, we say that a linear map
$\pi_{\epsilon}: \R^3 \longrightarrow \R^2$ is an \emph{orthogonal projection
along $\epsilon$} if $\ker{\pi_{\epsilon}} = \left\langle \epsilon
\right\rangle$ and $\pi_{\epsilon}$ is an isometry on $\left\langle \epsilon
\right\rangle^{\perp}$. Moreover we ask that the preimages of the standard
bases of~$\R^2$ lying on~$\langle \varepsilon \rangle^{\perp}$ form,
together with~$\varepsilon$, a positively oriented bases. Note that in this way
$\pi_{\epsilon}$ is well-defined only up to direct Euclidean isometry of the
image, namely rotations around the origin.
\end{definition}

\begin{remark}
\label{remark:dot_product}
	From the definition of~$\gamma$ given in the proof of 
Lemma~\ref{lemma:bijection_S2_conic} we see that the vectors~$\varepsilon'$ 
and~$\varepsilon''$ we form starting from $\varepsilon \in S^2$ satisfy
$\left(\begin{smallmatrix} \varepsilon & \varepsilon' & \varepsilon''
\end{smallmatrix}\right) \in \mathrm{SO}(3, \R)$. One can check that
$\left(\begin{smallmatrix} \varepsilon' & \varepsilon''
\end{smallmatrix}\right)^{t}$ gives the matrix of $\pi_{\varepsilon}$, the
orthogonal projection along~$\varepsilon$. Identifying~$\R^2$ with the complex
plane~$\C$ one finds that if $\gamma(\varepsilon) = (x:y:z)$ and $A = (p,q,r)$
is a point in~$\R^3$, then $\pi_{\varepsilon}(A)$ is given, as a point in~$\C$,
by $p \tth x + q \tth y + r \tth z$. If now we change the representative 
of~$\gamma(\varepsilon)$, this modifies the image under the orthogonal 
projection by possibly a rotation and a dilation.
\end{remark}

\smallskip
Hence taking into account Remark~\ref{remark:dot_product} we can
realize the orthogonal projection~$\pi_{\varepsilon}$ as the dot product
\[ \left\langle (x,y,z), \cdot \right\rangle: \R^3 \longrightarrow \C \] 
where $(x:y:z)$ is any representative of~$\gamma(\varepsilon)$ with Hermitian
norm equal to~$\sqrt{2}$. Thus we can view any orthogonal projection of~$n$
points $\vec{A} = (A_1, \ldots, A_n)$ as an $n$-tuple of points in~$\C$. Then
$\pi_{\varepsilon}(\vec{A})$ can be seen as one single point in $\C^n$. We can
use the embedding $\C \hookrightarrow \p^1_{\C}$ sending $z$ to~$(z:1)$ to
identify $\pi_{\varepsilon}(\vec{A})$ with a point in~$\left( \p^1_{\C}
\right)^n$. We will extensively use this fact in the definition of the M\"obius
camera, and for proving some of its properties. 

\medskip
In order to understand what should be the codomain of our photographic map we 
start with the following known result (see, for example, \cite{JonesSingerman}, 
Chapter~2, Sections~2.1 and~2.2):
\begin{proposition}
\label{proposition:moebius_automorphisms}
 $\big\{ \mathrm{M\ddot{o}bius\ transformations} \big\} \cong \p\mathrm{GL}(2,\C) \cong \mathrm{Aut}(\p^1_{\C})$.
\end{proposition}
Due to Proposition~\ref{proposition:moebius_automorphisms} we can use the
natural action of~$\p\mathrm{GL}(2,\C)$ on~$\left( \p^1_{\C} \right)^n$ to
express that two orthogonal projections are M\"obius equivalent. This leads us
to the following definition, which we denote as temporary because unfortunately,
despite the fact that it looks all-embracing and clean, it will not be very
useful for us.
\begin{tmpdefinition}
  Let $\vec{A}$ be a finite set of distinct points in~$\R^3$ and $\epsilon \in 
S^2$. The \emph{M\"obius picture of}~$\vec{A}$ \emph{along}~$\epsilon$ is the 
equivalence class, under the action of~$\p\mathrm{GL}(2,\C)$, of any 
orthogonal parallel projection of~$\vec{A}$ along the direction~$\epsilon$, 
considered as an $n$-tuple of points in~$\p^1_{\C}$.
\end{tmpdefinition}
\begin{remark}
	We notice that the concept of M\"obius picture is well-defined. In fact,
although the choice of different orthogonal projections along the same direction
determines different points of $\left( \p^1_{\C} \right)^n$, they all differ by
a M\"obius transformation (given by a rotation, as mentioned in the end of
Definition~\ref{definition:orthogonal_projection}), so their equivalence class
is the same.
\end{remark}
\begin{remark}
  Since the group of M\"obius transformations is $3$-transitive, then all
configurations of~$n$ points in~$\R^2$ are M\"obius equivalent for $1 \leq n
\leq 3$, so M\"obius photogrammetry can be reasonably approached only if $n
\geq 4$.
\end{remark}
The problem with the Temporary Definition is that it involves an object, the
quotient of~$\left( \p^1_{\C} \right)^n$ by the action of~$\p\mathrm{GL}(2,\C)$,
which does not have good geometric properties, namely it does not have a natural
structure of an algebraic variety. In 
Subsection~\ref{photogrammetry:camera:embedding} we will see that we
can obtain a much better object at a fair price, and we will focus on the case
which is most interesting for us, namely~$n = 5$. On the other hand, this
operation has a cost: we will not be able to define M\"obius pictures for any
arbitrary configuration of points (see Footnote~\ref{footnote:condition}), but 
we need to put some restrictions. However, we will see that the concept of 
M\"obius camera is meaningful for any configuration of points (see 
Remark~\ref{remark:regular_rational}). 

\section{The M\"obius camera}
\label{photogrammetry:camera}

\subsection{A projective embedding of the moduli space of~$5$ points 
in~$\p^1_{\C}$}
\label{photogrammetry:camera:embedding}

Geometric Invariant Theory tells us that, in order to obtain our desired set of
equivalence classes under the action of~$\p\mathrm{GL}(2,\C)$, which we denote
by~$M_5$\footnote{In the literature this object is usually denoted 
by~$\mathcal{M}_{0,5}$, since it is the moduli space of genus~0 smooth curves
with~5 marked points, but here we will always omit the index~$0$.}, we cannot
consider the equivalence classes of all $5$-tuples, but we have to restrict to
an open subset of $\left(\p^1_{\C}\right)^5$ (for an introduction to this 
topic, see~\cite{Dolgachev2003}, in particular Chapter~$6$, 
or~\cite{MumfordFogartyKirwan}, in particular Chapters~$0$ and~$1$). Therefore 
we will be forced to impose some conditions on the vector~$\vec{A}$ of points 
in~$\R^3$, in order to ensure that it is possible to define its M\"obius 
picture along a given direction. After accepting this limitation one can 
construct an embedding of the quotient in projective space defined by invariants 
of the $5$-tuple\footnote{For $n = 4$, there are two invariants, defining an 
open embedding $M_4 \hookrightarrow \p^1_{\C}$; the quotient of the two 
projective invariants is an absolute invariant, the cross ratio. This is in 
accordance with the fact that the projective equivalence of two $4$-tuples of 
points in~$\p^1_{\C}$ is completely determined by their cross ratio.}. It is 
possible to embed~$M_5$ as a quintic surface in~$\p^5_{\C}$: 
in~\cite{HowardMillisonSnowdenVakil} it is explained a possible way to determine
this surface and the quotient map~$\left( \p^1_{\C} \right)^5 \dashrightarrow
M_5$. We briefly describe the procedure:
\begin{itemize}
	\item[1.]
		Consider a convex pentagon~$P$ in the plane, and construct all plane
undirected multigraphs without loops whose set of vertices coincides with the
set of vertices of~$P$, and which satisfy the following conditions:	
\begin{itemize}
			\item[$\cdot$]
				edges are given by segments;
			\item[$\cdot$]
				any two edges do not intersect;
			\item[$\cdot$]
				the valency of every vertex is~2;
		\end{itemize}
		There are exactly 6 of these graphs, showed in Figure~\ref{figure:graphs}.
		\begin{figure}[ht!]
			\begin{tabular}{c|c|c}
				\subfigure{
					\begin{overpic}[width=0.2\textwidth]{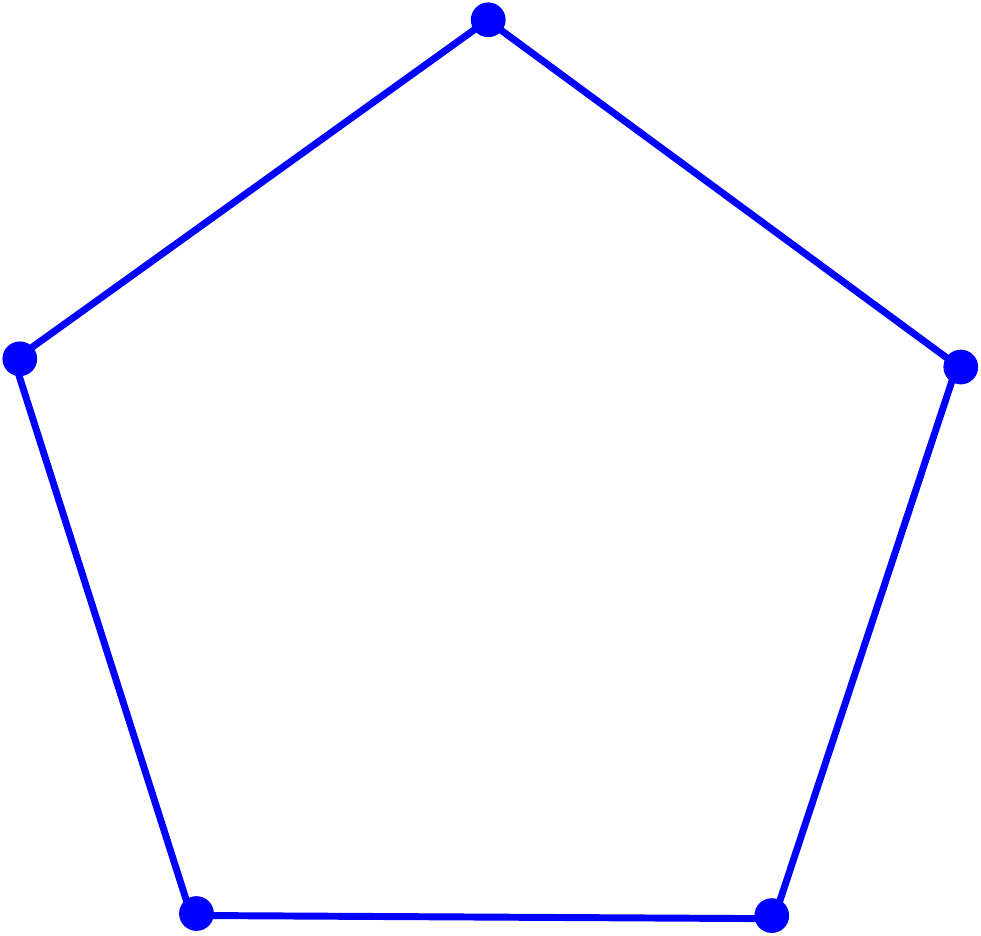}
						\begin{small}
							\put(46.5,97){$1$}
							\put(100,60){$2$}
							\put(80,-10){$3$}
							\put(12,-10){$4$}
							\put(-7,60){$5$}
						\end{small} 
					\end{overpic}
				} \hspace{0.4cm} & \hspace{0.4cm}
				\subfigure{
					\begin{overpic}[width=0.2\textwidth]{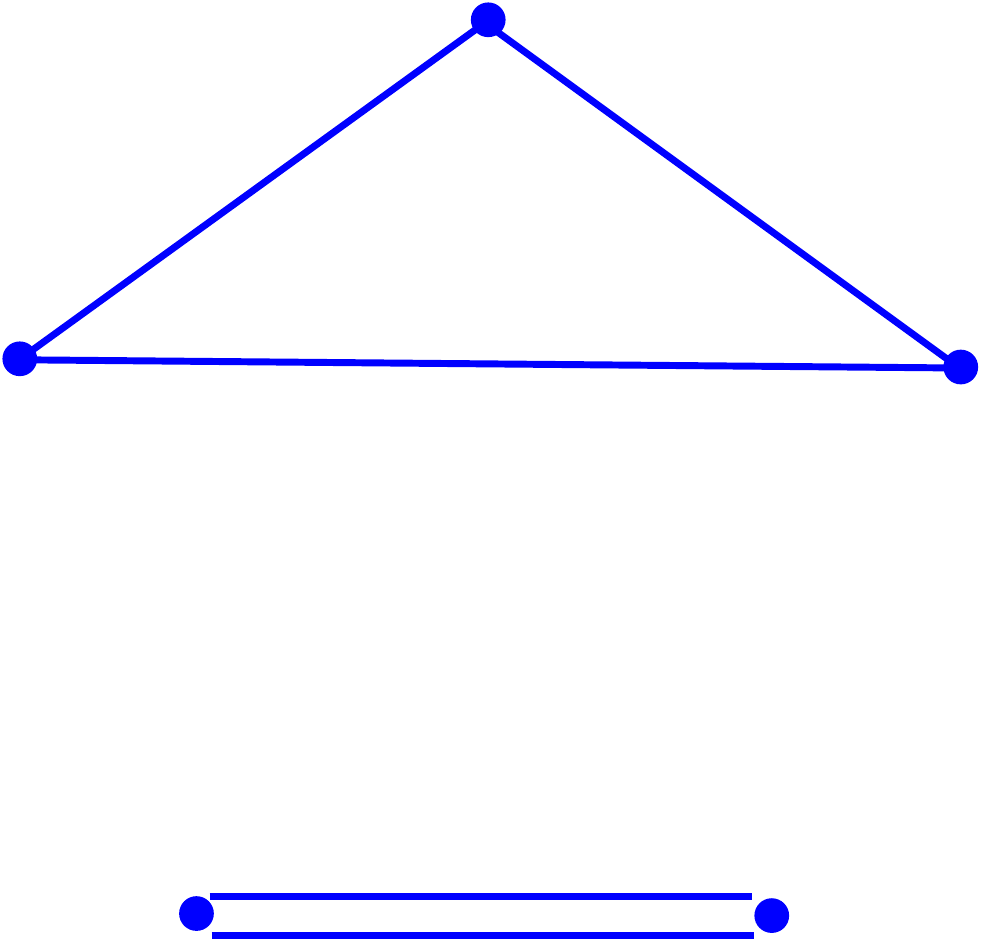}
						\begin{small}
							\put(46.5,97){$1$}
							\put(100,60){$2$}
							\put(80,-10){$3$}
							\put(12,-10){$4$}
							\put(-7,60){$5$}
						\end{small} 
					\end{overpic}
				} \hspace{0.4cm} & \hspace{0.4cm}
				\subfigure{
					\begin{overpic}[width=0.2\textwidth]{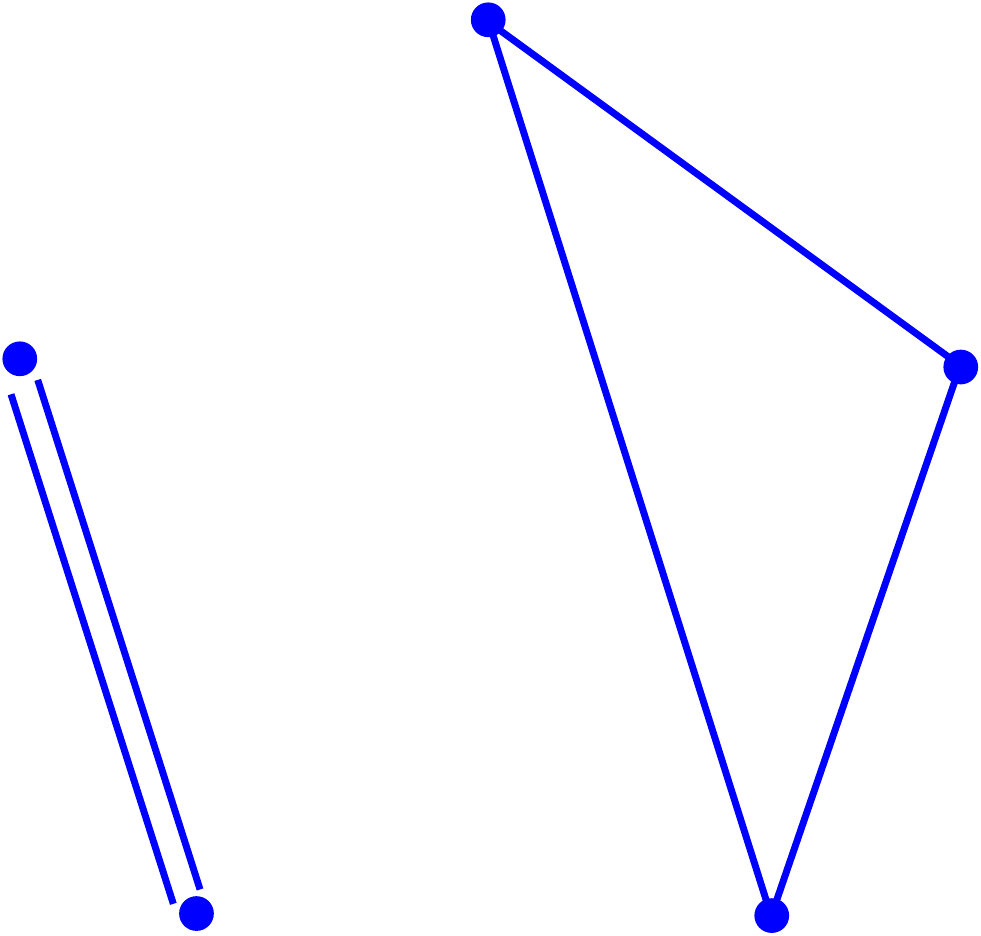}
						\begin{small}
							\put(46.5,97){$1$}
							\put(100,60){$2$}
							\put(80,-10){$3$}
							\put(12,-10){$4$}
							\put(-7,60){$5$}
						\end{small} 
					\end{overpic}
				} \\ [4mm] \hline & & \\ 
				\subfigure{
					\begin{overpic}[width=0.2\textwidth]{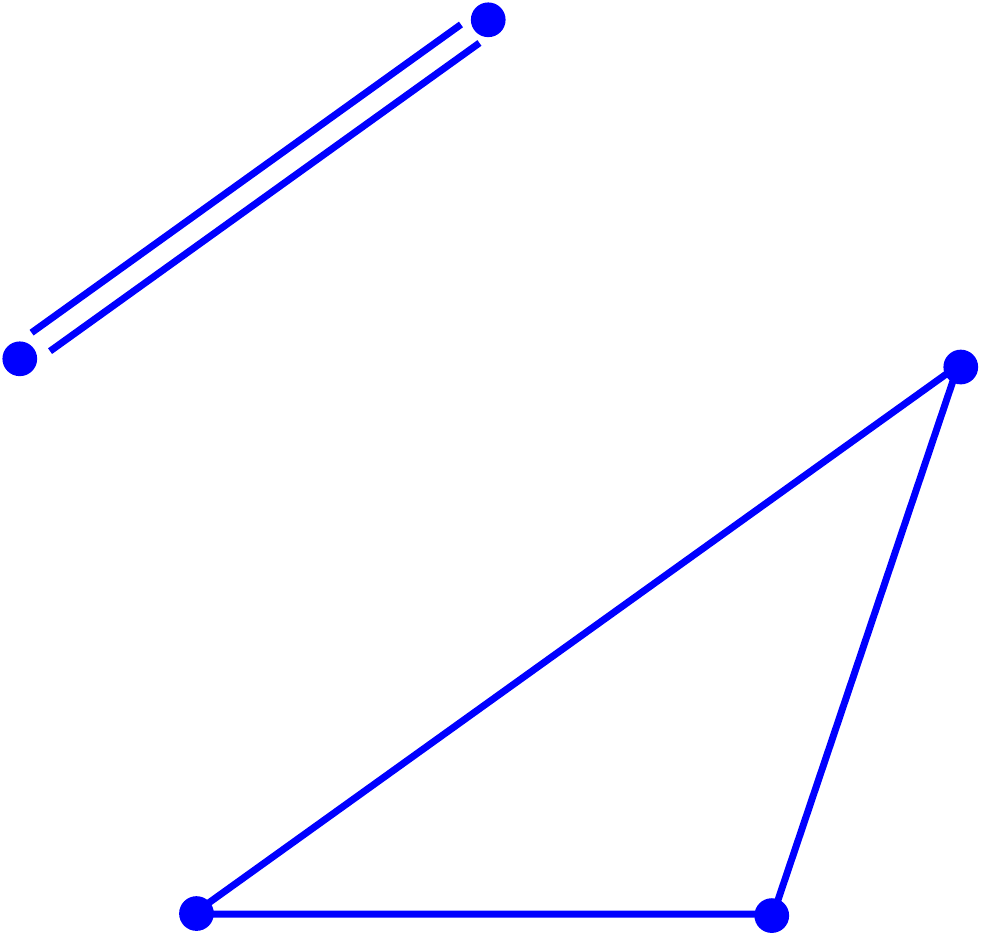}
						\begin{small}
							\put(46.5,97){$1$}
							\put(100,60){$2$}
							\put(80,-10){$3$}
							\put(12,-10){$4$}
							\put(-7,60){$5$}
						\end{small} 
					\end{overpic}
				} \hspace{0.4cm} & \hspace{0.4cm}
				\subfigure{
					\begin{overpic}[width=0.2\textwidth]{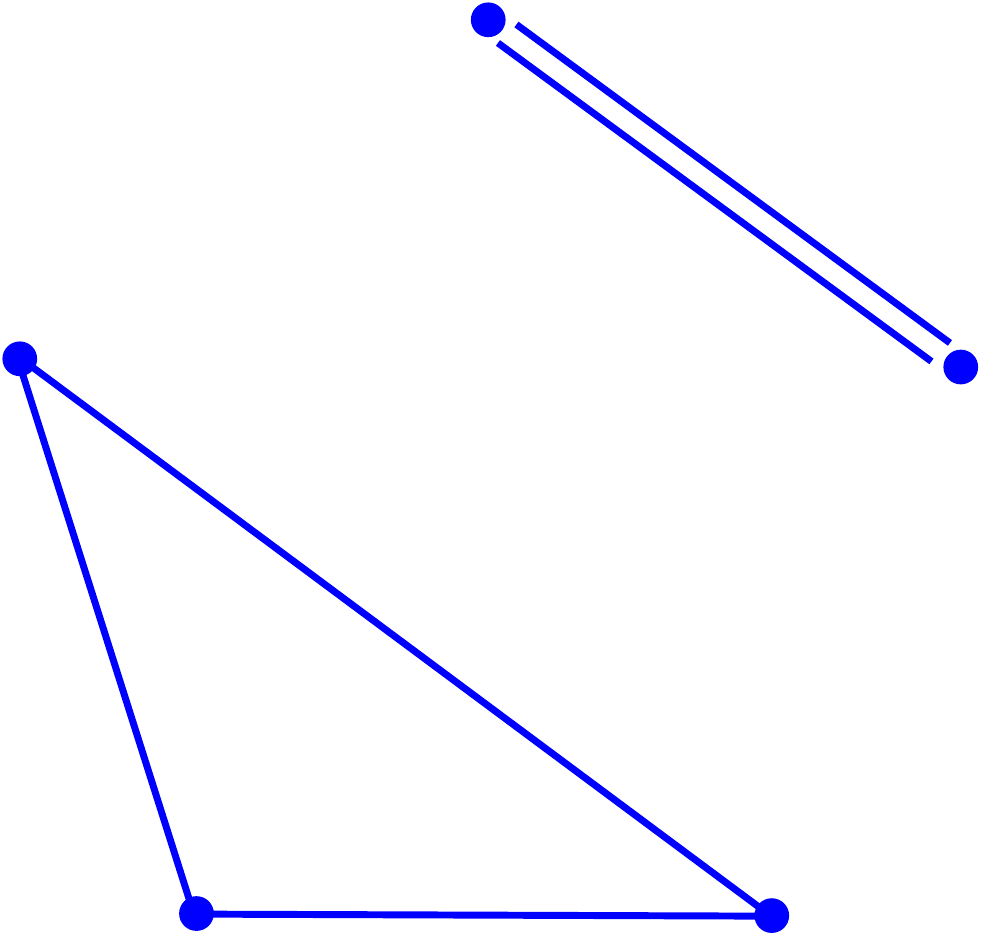}
						\begin{small}
							\put(46.5,97){$1$}
							\put(100,60){$2$}
							\put(80,-10){$3$}
							\put(12,-10){$4$}
							\put(-7,60){$5$}
						\end{small} 
					\end{overpic}
				} \hspace{0.4cm} & \hspace{0.4cm}
				\subfigure{
					\begin{overpic}[width=0.2\textwidth]{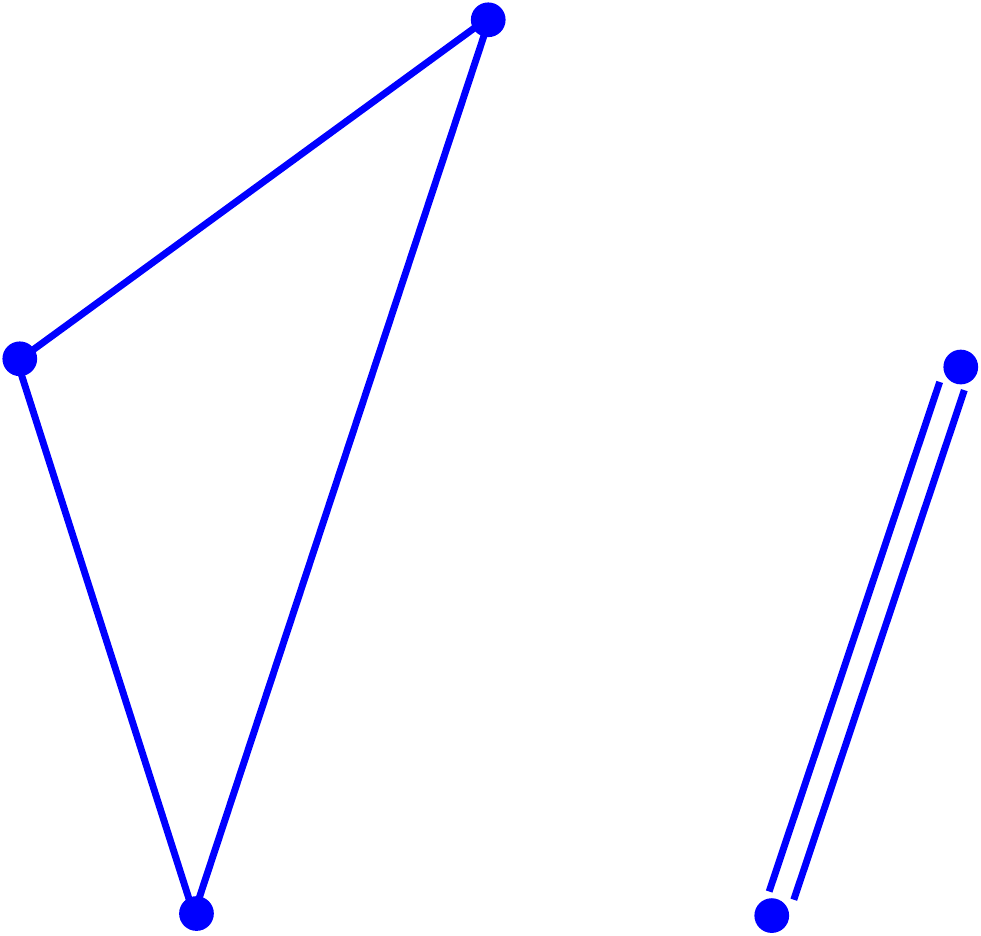}
						\begin{small}
							\put(46.5,97){$1$}
							\put(100,60){$2$}
							\put(80,-10){$3$}
							\put(12,-10){$4$}
							\put(-7,60){$5$}
						\end{small} 
					\end{overpic}
				}
			\end{tabular}
		\caption{The only six planar undirected multigraphs without loops with 
vertices on a regular pentagon, valency~$2$ and non-intersecting edges.}
		\label{figure:graphs}
		\end{figure}
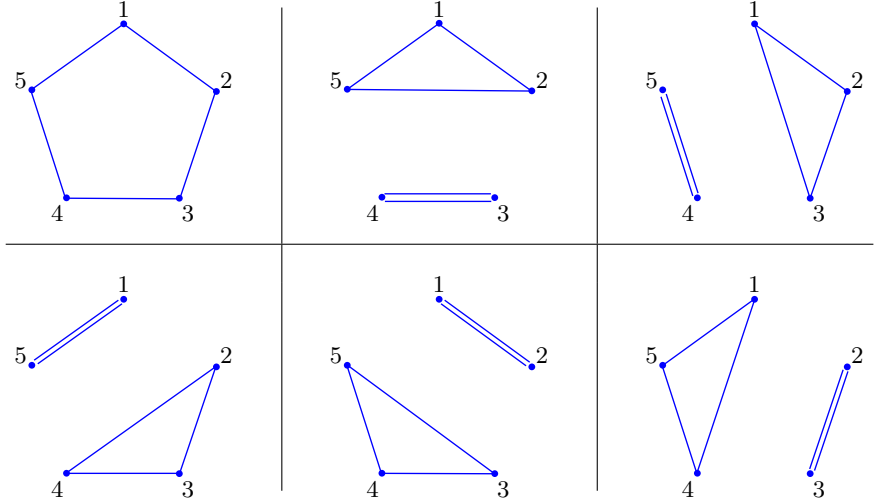
	\item[2.]
		Associate to each graph~$G = (E,V)$ a homogeneous polynomial in the 
coordinates $\{(a_i:b_i)\}$ of~$\left( \p^1_{\C} \right)^5$ according to the 
following rules:
		\begin{itemize}
			\item[i.]
				for every edge $e \in E$, $e = (i,j)$ with $i < j$, define
				\[ \varphi_{e} \; = \; a_i \tth b_j - a_j \tth b_i \]
			\item[ii.]
				set 
				\[ \varphi_{G} \; = \; \prod_{e \in E} \varphi_{e} \]
		\end{itemize}
		For example the polynomial associated to the first graph in 
		Figure~\ref{figure:graphs} is
		\[ \hspace{0.8cm} \resizebox{0.85\textwidth}{!}{$\varphi_0 = (a_1 \tth b_2 
- a_2 \tth b_1) (a_2 \tth b_3 - a_3 \tth b_2) (a_3 \tth b_4 - a_4 \tth b_3)(a_4 
\tth b_5 - a_5 \tth b_5)(a_1 \tth b_5 - a_5 \tth b_1)$} \]
	\item[3.]
		These polynomials determine a rational map $\varphi: \left( \p^1_{\C} \right)^5 \dashrightarrow \p^{5}_{\C}$.
	\item[4.]
		Consider the open set
		\[ \mathcal{U} \; = \; \big\{ (m_1, \ldots, m_5) \, : \; \textrm{no three of
the  } m_i \textrm{ coincide} \big\} \; \subseteq \; \left( \p^1_{\C} \right)^5
\]
	\item[5.]
		We have that $\mathrm{image} \left(\varphi_{|_\mathcal{U}}\right) = M_5$. It turns out that if we take coordinates $t, x_1, \ldots, x_5$ in $\p^5_{\C}$, then the equations for $M_5$ are:
		\[ x_{i-2} \tth x_{i+2} \; = \; t \tth x_i + t^2 \quad \quad \forall i \in \{1, \ldots, 5\} \]
		where the indices are taken modulo~$5$.
\end{itemize}
\begin{remark}
\label{remark:description_M_5}
  It is known that~$M_5$ is a Del Pezzo surface of degree~$5$: exactly~$10$
lines lie on such a surface and they correspond to equivalence classes
of $5$-tuples $(m_1, \ldots, m_5)$ for which at least two points coincide. We
denote by~$L_{ij}$ the line in~$M_5$ corresponding to classes for which $m_i =
m_j$.
\end{remark}
\begin{remark}
\label{remark:real_points_M_5}
  The variety~$M_5$ has a canonical real structure, inherited from the real
structure of the projective line. An equivalence class is real (namely, it is a
fixed point for the anti-holomorphic involution on~$M_5$) if and only if it can
be represented by a $5$-tuple of real points in~$\p^1_{\C}$. Equivalently, the
points are colinear or cocircular; in fact, we recall these two facts about
automorphisms of~$\p^1_{\C}$:
  \begin{itemize}
    \item[$\cdot$]
    M\"obius transformations map lines and circles to lines and circles;
    \item[$\cdot$]
    the action of M\"obius transformations is transitive on lines and circles.
  \end{itemize}
  Since a $5$-tuple is given by real points in~$\p^1_{\C}$ if and only if the
corresponding points in~$\R^2$ lie on a line, the claim follows from the
previous two considerations. 
\end{remark}

\subsection{Definition of the M\"obius camera}
\label{photogrammetry:camera:definition}

From this construction we infer the condition we have to impose on the 
vector~$\vec{A}$ of points in~$\R^3$ so that we can speak of a M\"obius
picture along an arbitrary direction\footnote{\label{footnote:condition} This 
is the condition on configurations mentioned at the end of
Section~\ref{photogrammetry:setup}.}: no~$3$ points among the $\{A_i\}$ should
be aligned. In this way for every $\epsilon \in S^2$ we will have that
$\pi_{\epsilon}(\vec{A})$ lies on $\mathcal{U}$, hence its equivalence class is
a well-defined element of~$M_5$.
\begin{definition}
\label{definition:moebius_picture}
   Let~$\vec{A}$ be a vector of~$5$ points in~$\R^3$ and $\epsilon \in S^2$. 
Suppose that no three points of~$\vec{A}$ lie on a line parallel to~
$\varepsilon$. The \emph{M\"obius picture of}~$\vec{A}$ \emph{along}~$\epsilon$ 
is the point in $M_5$ given by the equivalence class (under the action of 
$\p\mathrm{GL}(2,\C)$) of any orthogonal parallel projection of~$\vec{A}$ along 
the direction~$\epsilon$, considered as an $n$-tuple of points in $\p^1_{\C}$. 
\end{definition}

Our next task is to define the notion of M\"obius camera, namely a function 
which, once fixed a vector of points, takes a direction $\epsilon \in S^2$ and 
associates to it the M\"obius picture of the points along that direction. 

\begin{definition}
\label{definition:moebius_camera}
  Let $\vec{A}$ be a vector of~$5$ points in~$\R^3$. The \emph{M\"obius 
camera}, or \emph{photographic map}, for~$\vec{A}$ is the morphism of varieties 
given by:
	\[ 
		\begin{array}{rccc}
			f_{\vec{A}}: & C & \longrightarrow & M_5 \subseteq \p^5_{\C} \\
			& c & \mapsto &  \;\; \varphi \Big( \, \big( \pi_{\epsilon}(A_1) : 1 \big), \, \ldots \, , \big( \pi_{\epsilon}(A_5) : 1 \big) \, \Big) \\
			& & & = \varphi \Big( \, \big( \! \left\langle c, A_1 \right\rangle : 1 \big), \, \ldots \, , \big( \! \left\langle c, A_5 \right\rangle : 1 \big) \, \Big)
		\end{array}
	\]
	where~$C$ is the curve $\{ x^2 + y^2 + z^2 = 0\} \subseteq \p^2_{\C}$, and $c
\leftrightarrow \epsilon$ under the bijection~$\gamma$ established in 
Lemma~\ref{lemma:bijection_S2_conic}; moreover the map $\varphi: \left( 
\p^1_{\C} \right)^5 \dashrightarrow M_5$ is the quotient map defined in 
Subsection~\ref{photogrammetry:camera:embedding}. We denote by 
$\tilde{f}_{\vec{A}}$ the precomposition of~$f_{\vec{A}}$ by the 
parametrization of~$C$ described in
Remark~\ref{remark:three_descriptions_S2}.
\end{definition}
\begin{remark}
\label{remark:regular_rational}
	The M\"obius picture of~$\vec{A}$ cannot be defined for those~$c \in C$ such
that there exist three points in~$\vec{A}$ lying on a line parallel to the
direction defined by~$c$. Since the points~$c \in C$ for
which the M\"obius picture is defined form an open subset of~$C$, then the 
map~$f_{\vec{A}}$ is a priori a rational one. However, since $C$ is a smooth
curve, then $f_{\vec{A}}$ extends to a regular map, namely it is defined also
on the points which do not admit a M\"obius picture. In algebraic terms, the
polynomials defining the function have a common factor vanishing at those
points, which can be canceled.
\end{remark}
\begin{remark}
  The map~$\varphi$ is given by homogeneous polynomials of degree~$5$ in the
coordinates $\{ (a_i: b_i) \}$ of $\left( \p^1_{\C} \right)^5$. Hence
$\tilde{f}_{\vec{A}}$ is given by homogeneous polynomials of degree~$10$ in the
coordinates~$(s:t)$ of~$\p^1_{\C}$.
\end{remark}

\subsection{Properties of the M\"obius camera}
\label{photogrammetry:camera:properties}

The following lemmata describe the behavior of the image of a M\"obius camera
depending on the geometry of the vector~$\vec{A}$.
\begin{lemma} 
\label{lemma:5plane}
	Let $\vec{A} = (A_1, \ldots, A_5)$ be a $5$-tuple of coplanar points which are 
not collinear. Then the photographic map $f_{\vec{A}}: C \longrightarrow M_5$ is 
$2:1$ to a rational curve of degree~$5$, $4$, $3$, or~$2$ in~$M_5$.
\end{lemma}
\begin{proof} 
	Suppose that the~$A_i$ are coplanar: then, after a suitable change of
coordinates, we can assume that $A_i = (p_i, q_i,0)$ for every~$i$, and in this
case the photographic map~$f_{\vec{A}}$ factors through the restriction to~$C$
of the projection $\tau_{x,y}: \p^2_{\C} \longrightarrow \p^1_{\C}$ sending
$(x:y:z) \mapsto (x:y)$, which is a $2:1$~map. Hence we get
	\[ \xymatrix{ C \ar[rr]^-{f_{\vec{A}}} \ar[dr]_-{\tau_{x,y}} & & M_5 \\ &
\p^1_{\C} \ar[ru]_- {g_{\vec{A}}} } \]
	If we show that $g_{\vec{A}}$ is birational, then $f_{\vec{A}}$ is $2:1$. 
	The map $g_{\vec{A}}$ is given by~$6$ components, each of which is the 
product of five linear polynomials in~$x$ and~$y$. Each of these polynomials is 
of the form $G_{ij} = x (p_i - p_j) + y (q_i - q_j)$. Hence the components 
of~$g_{\vec{A}}$ have the following structure:
	\begin{equation}
	\label{equation:structure_plane_g}
		\begin{array}{rcccccc}
			\left( g_{\vec{A}} \right)_{0} & = & G_{12} & G_{23} & G_{34} & G_{45} & G_{15} \\
			\left( g_{\vec{A}} \right)_{1} & = & G_{12} & G_{25} & G_{15} & G_{34} & G_{34} \\
			\left( g_{\vec{A}} \right)_{2} & = & G_{12} & G_{23} & G_{13} & G_{45} & G_{45} \\
			\left( g_{\vec{A}} \right)_{3} & = & G_{23} & G_{34} & G_{24} & G_{15} & G_{15} \\
			\left( g_{\vec{A}} \right)_{4} & = & G_{34} & G_{45} & G_{35} & G_{12} & G_{12} \\
			\left( g_{\vec{A}} \right)_{5} & = & G_{14} & G_{45} & G_{15} & G_{23} & G_{23} \\
		\end{array}
	\end{equation}
	We notice that if the lines~$\overrightarrow{A_iA_j}$ 
and~$\overrightarrow{A_hA_k}$ are parallel, then $G_{ij}$ and $G_{hk}$ only 
differ by a scalar multiple. Since the~$5$ points are not collinear, only four 
configurations are allowed (after possibly relabeling the points), as shown in 
Figure \ref{figure:5plane}.
	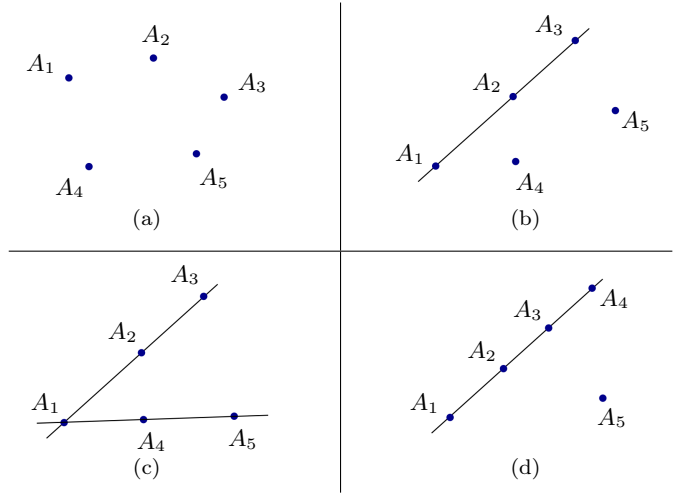
\begin{figure}[ht!]
		\begin{tabular}{c|c}
			\subfigure[][]{
			\label{figure:5plane-a}
				\begin{overpic}[width=0.25\textwidth]{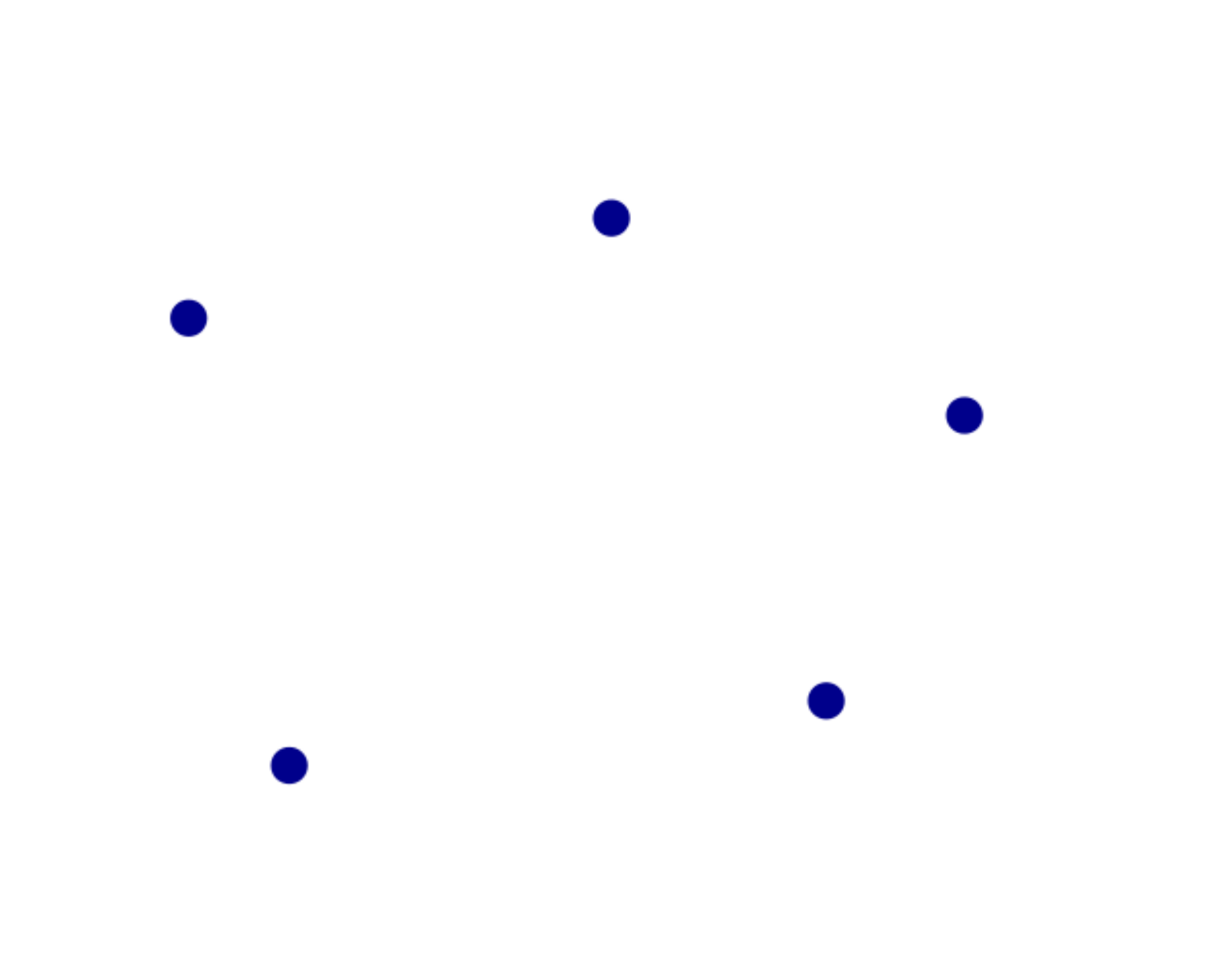}
					\begin{small}
						\put(-2,58){$A_1$}
						\put(46,70){$A_2$}
						\put(86,50){$A_3$}
						\put(10,6){$A_4$}
						\put(70,10){$A_5$}
					\end{small} 
				\end{overpic}
			} \hspace{0.5cm} & \hspace{0.5cm}
			\subfigure[][]{
			\label{figure:5plane-b}
				\begin{overpic}[width=0.25\textwidth]{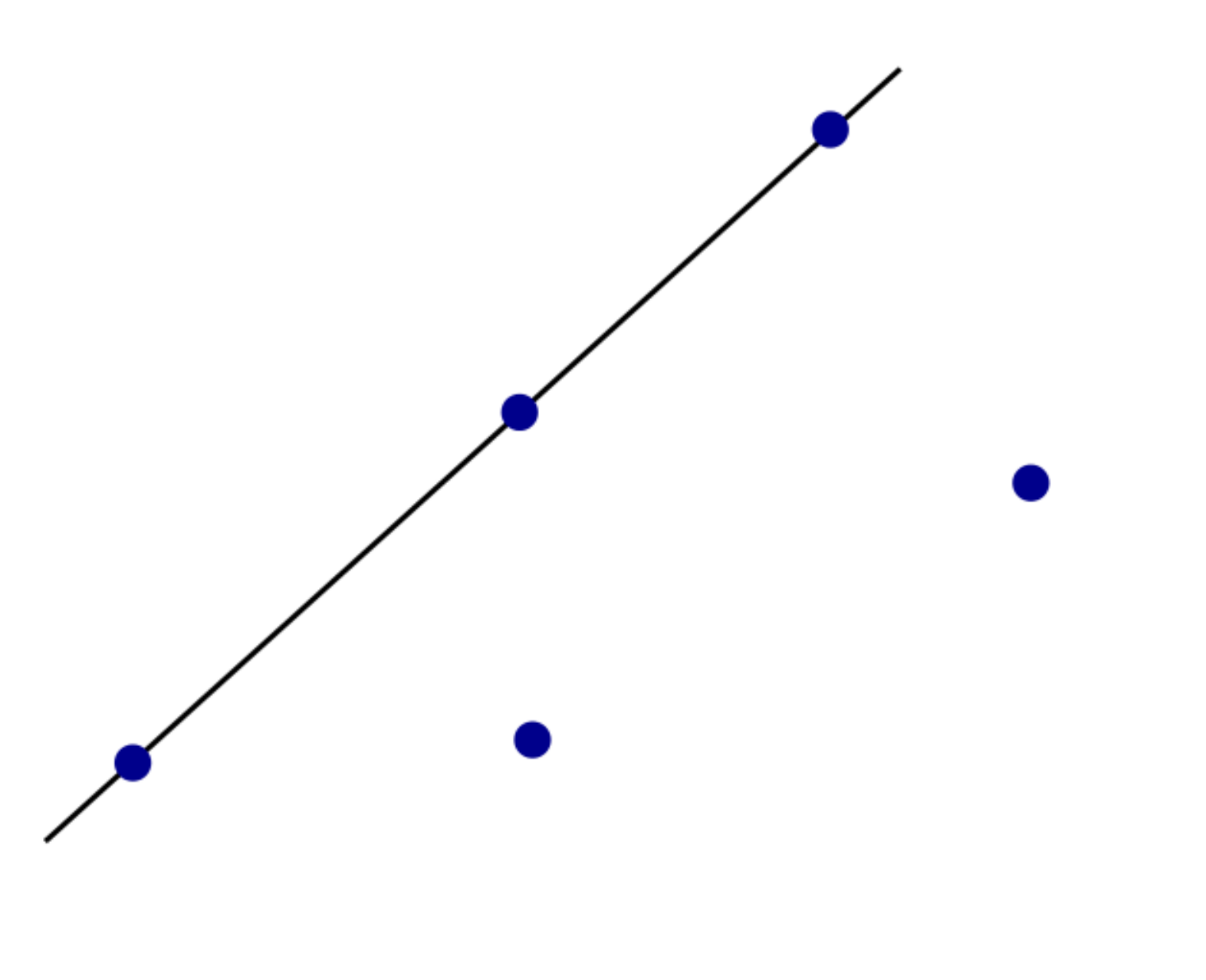}
					\begin{small}
						\put(-6,20){$A_1$}
						\put(26,50){$A_2$}
						\put(52,74){$A_3$}
						\put(44,8){$A_4$}
						\put(88,34){$A_5$}
					\end{small} 
				\end{overpic}
			} \\ [4mm] \hline
			\subfigure[][]{
			\label{figure:5plane-c}
				\begin{overpic}[width=0.25\textwidth]{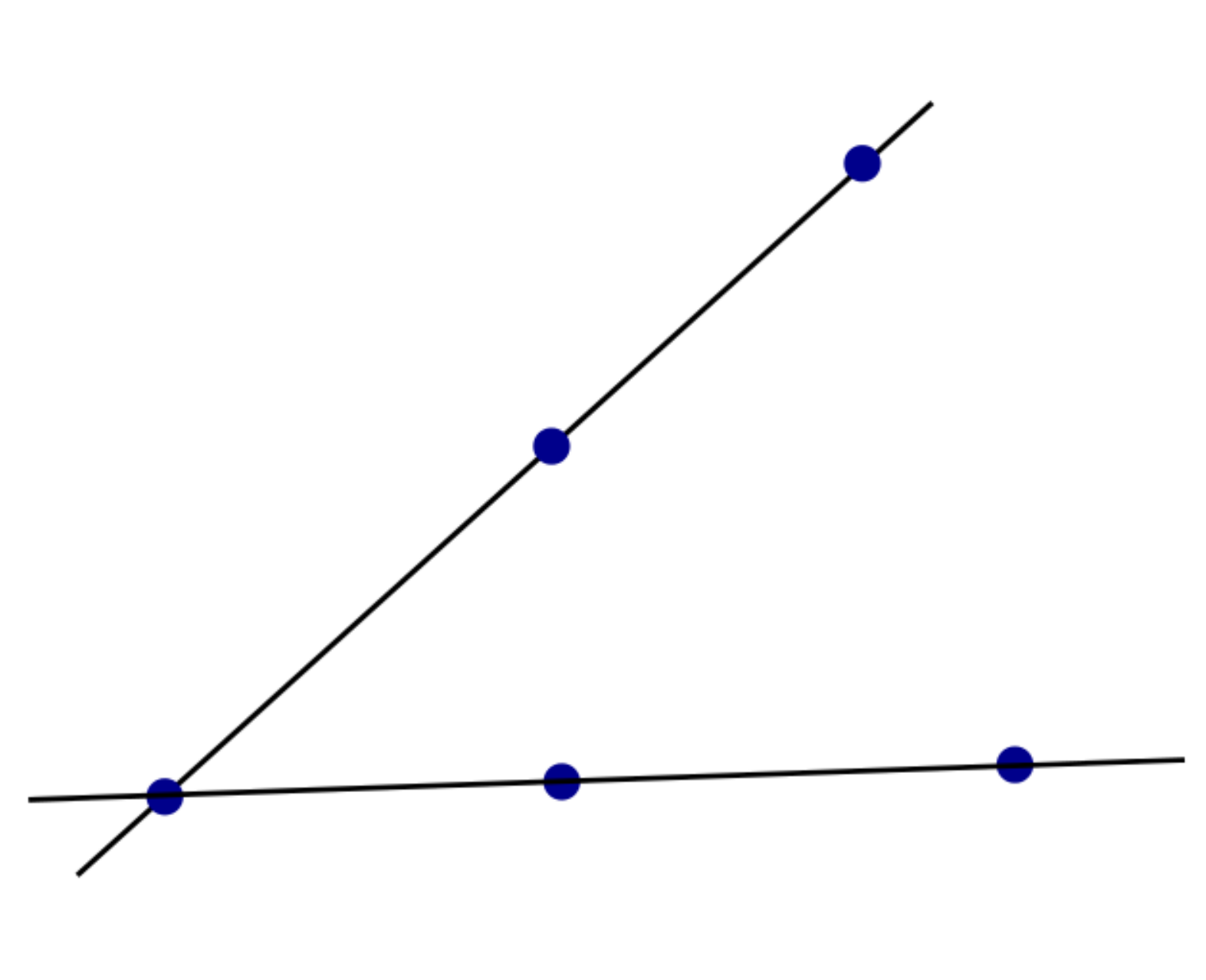}
					\begin{small}
						\put(0,20){$A_1$}
						\put(32,50){$A_2$}
						\put(58,74){$A_3$}
						\put(44,4){$A_4$}
						\put(82,5){$A_5$}
					\end{small} 
				\end{overpic}
			} \hspace{0.5cm} & \hspace{0.5cm}
			\subfigure[][]{
			\label{figure:5plane-d}
				\begin{overpic}[width=0.25\textwidth]{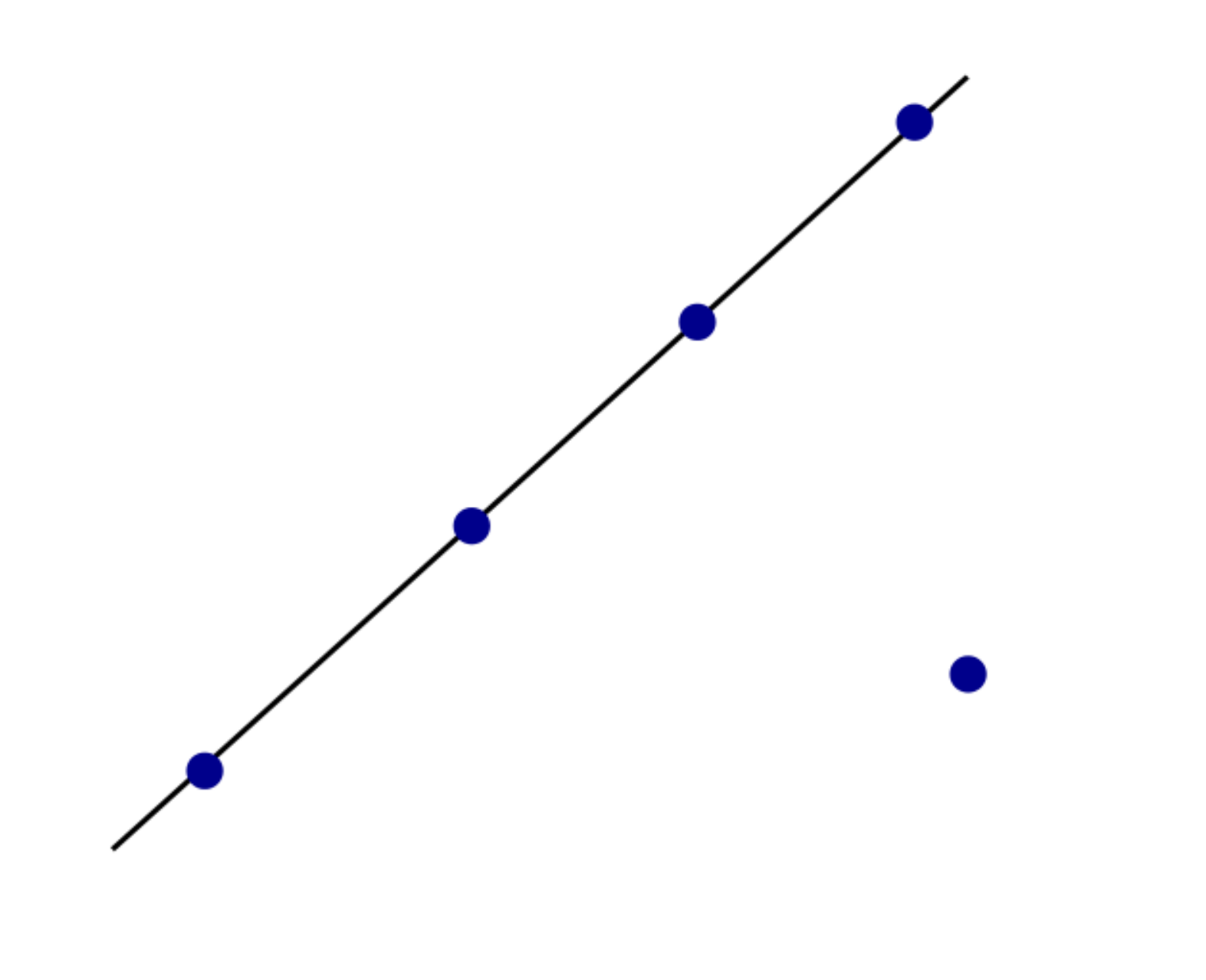}
					\begin{small}
						\put(0,20){$A_1$}
						\put(24,42){$A_2$}
						\put(43,59){$A_3$}
						\put(79,64){$A_4$}
						\put(78,14){$A_5$}
					\end{small} 
				\end{overpic}
			}
		\end{tabular}
	\caption{Possible configurations of~$5$ points in the plane: 
\subref{figure:5plane-a} no~$3$ points are aligned, \subref{figure:5plane-b} 
exactly~$3$ points are aligned, \subref{figure:5plane-c} $3+3$ points are 
collinear, \subref{figure:5plane-d} exactly~$4$ points are collinear.}
	\label{figure:5plane}
	\end{figure}
	\begin{description}
		\item[Case (a)]
			The components of~$g_{\vec{A}}$ do not have factors in common, so 
			\[ \deg{\left( g_{\vec{A}}(\p^1_{\C}) \right)} \cdot \deg{(g_{\vec{A}})} \; = \; 5 \]
			Hence either $g_{\vec{A}}$ is a birational map to a curve of degree $5$,
or it is a $5:1$ map to a line. If the image of~$g_{\vec{A}}$ were a line, then
because of Remark~\ref{remark:description_M_5} that line would coincide with
one of the~$10$ lines of~$M_5$. This would mean that whatever direction we use,
the projections of two points always coincide in any M\"obius picture, and this
is not possible. Hence this possibility must be ruled out, obtaining that
$g_{\vec{A}}$ is birational. 
		\item[Case (b)]
			Here $G_{12}$, $G_{23}$ and $G_{13}$ are equal up to scalar multiplication, so all the components have one factor in common, which can be removed. Hence 
			\[ \deg{\left( g_{\vec{A}}(\p^1_{\C}) \right)} \cdot \deg{(g_{\vec{A}})} \; = \; 4 \]
			this leading to three possibilities: $\deg{(g_{\vec{A}})} = 1$, $2$ 
or~$4$. The case when $\deg{(g_{\vec{A}})}$ is $4$ can be discarded as in 
Case~(a), so in order to prove the thesis we only have to consider the 
situation $\deg{(g_{\vec{A}})} = 2$. In this case the image of~$g_{\vec{A}}$ 
would be a conic, but from the general theory of Del Pezzo surfaces we have the 
following:

\smallskip
\noindent {\bf Claim.} $g_{\vec{A}}(\p^1_{\C})$ cannot be a conic. \\
{\it Proof.} It is well known that the surface~$M_5$ 
contains~$5$ families of conics, and every irreducible conic belongs exactly to 
one of them. These families arise in the following way: fix an index $i \in 
\{1, \ldots, 5\}$ and consider the map $M_5 \longrightarrow M_4 \cong 
\p^1_{\C}$ sending the equivalence class of $(m_1, \ldots, m_5)$ to the 
equivalence class of $(m_1, \ldots, m_{i-1}, m_{i+1}, \ldots, m_5)$,
namely we remove the~$i$-th point; the fibers of this 
map give one family of conics. From this description, recalling the definition
of the lines~$L_{ij}$ (see Remark~\ref{remark:description_M_5}), we see 
that the~$i$-th family of conics intersects only~$4$ lines, namely $L_{ij}$
for~$j \neq i$ (recall that $L_{ij} = L_{ji}$). On the other hand, by
inspecting our current situation, we see that the image $g_{\vec{A}}(\p^1_{\C})$
has to intersect the lines $L_{14}$, $L_{15}$, $L_{24}$, $L_{25}$, $L_{34}$ and 
$L_{35}$ (in general it will also intersect the line $L_{45}$, but this does
not happen if $\overrightarrow{A_1A_3}$ and $\overrightarrow{A_4A_5}$ are
parallel). Thus $g_{\vec{A}}(\p^1_{\C})$ cannot be one of the conics in~$M_5$.

\smallskip
\noindent Hence $g_{\vec{A}}$ can only be birational to a curve of 
degree~$4$.

		\item[Case (c)]
			Here $G_{12}$, $G_{23}$ and $G_{13}$ are equal up to scalar multiplication 
and the same for $G_{14}$, $G_{45}$ and $G_{15}$. One can check that all 
components have two factors in common. Thus, considering what we did in 
Case~(a), the only possible situation is the one in which $g_{\vec{A}}$ is 
birational to a curve of degree~$3$.
		\item[Case (d)]
			In this case $G_{12}$, $G_{23}$, $G_{13}$, $G_{24}$, $G_{34}$ and $G_{14}$ 
are equal up to scalar multiplication. One deduces that all components have 
three factors in common and so analogously as in Case~(c) we have that 
$g_{\vec{A}}$ is birational to a curve of degree~$2$. \qedhere
	\end{description}
\end{proof}
\begin{lemma} 
\label{lemma:5space}
	Let $\vec{A} = (A_1, \ldots, A_5)$ be a $5$-tuple of points. If the~$\{A_i\}$
are not coplanar, then the photographic map $f_{\vec{A}}: C \longrightarrow M_5$
is birational to a rational curve of degree~$10$ or~$8$ in~$M_5$.
\end{lemma}
\begin{proof}
	We argue as in the proof of Lemma~\ref{lemma:5plane}: if we write~$H_{ij}$ for
the linear polynomial $x (p_i - p_j) + y (q_i - q_j) + z (r_i - r_j)$, then the
components of $f_{\vec{A}}$ have the same structure as described by 
Equation~(\ref{equation:structure_plane_g}), where we replace~$G_{ij}$ 
by~$H_{ij}$. Since the~$\{A_i\}$ are not coplanar, we can have only three 
possibilities (after a possible relabeling of the points), showed in 
Figure~\ref{figure:5space}.
	\begin{figure}[ht!]
		\begin{tabular}{c|c|c}
			\subfigcapskip = 6pt
			\subfigure[][]{
			\label{figure:5space-a}
				\begin{overpic}[width=0.28\textwidth]{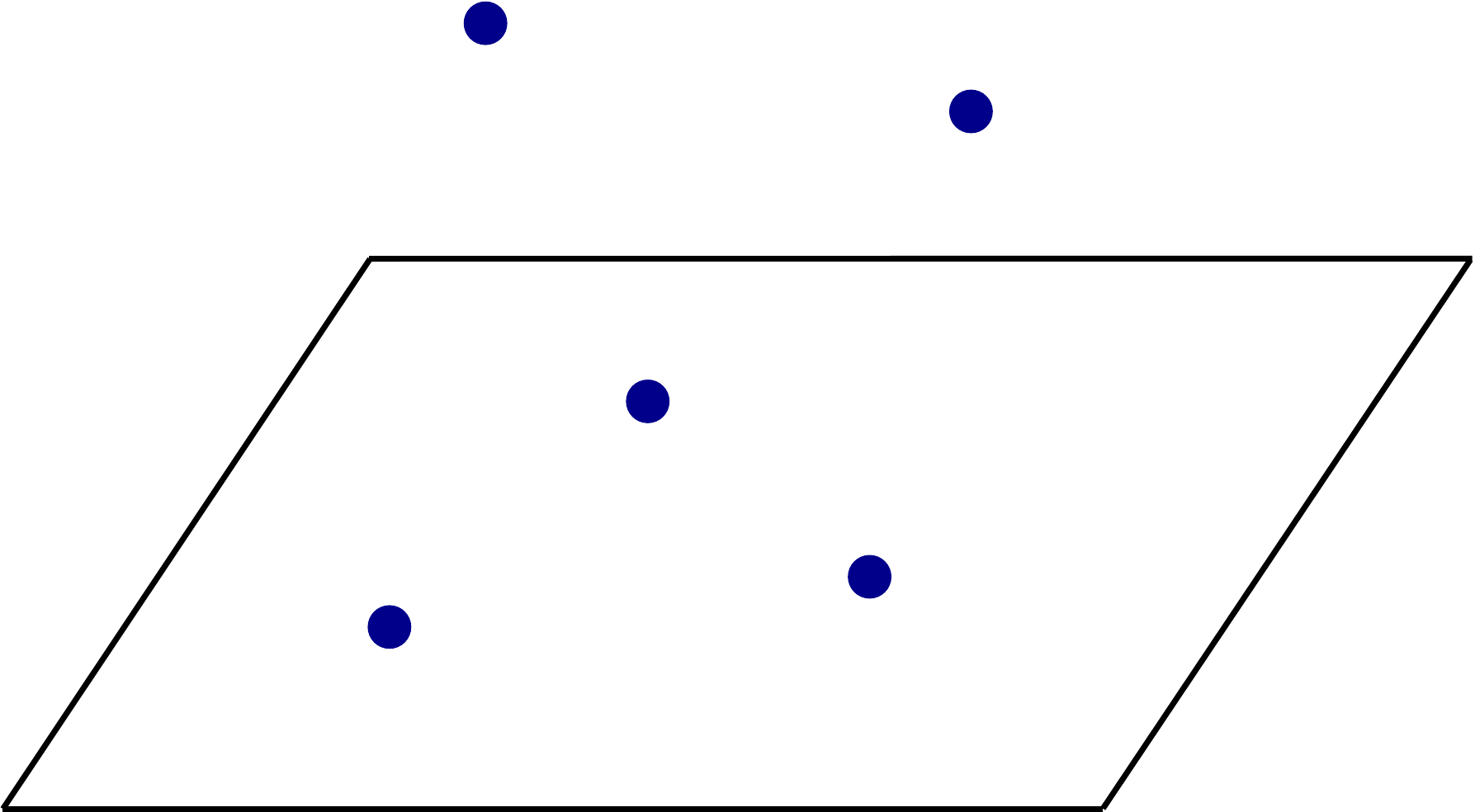}
					\begin{small}
						\put(12,4){$A_1$}
						\put(26,26){$A_2$}
						\put(58,6){$A_3$}
						\put(70,48){$A_4$}
						\put(16,54){$A_5$}
					\end{small} 
				\end{overpic}
			} &
			\subfigcapskip = 6pt
			\subfigure[][]{
			\label{figure:5space-b}
				\begin{overpic}[width=0.28\textwidth]{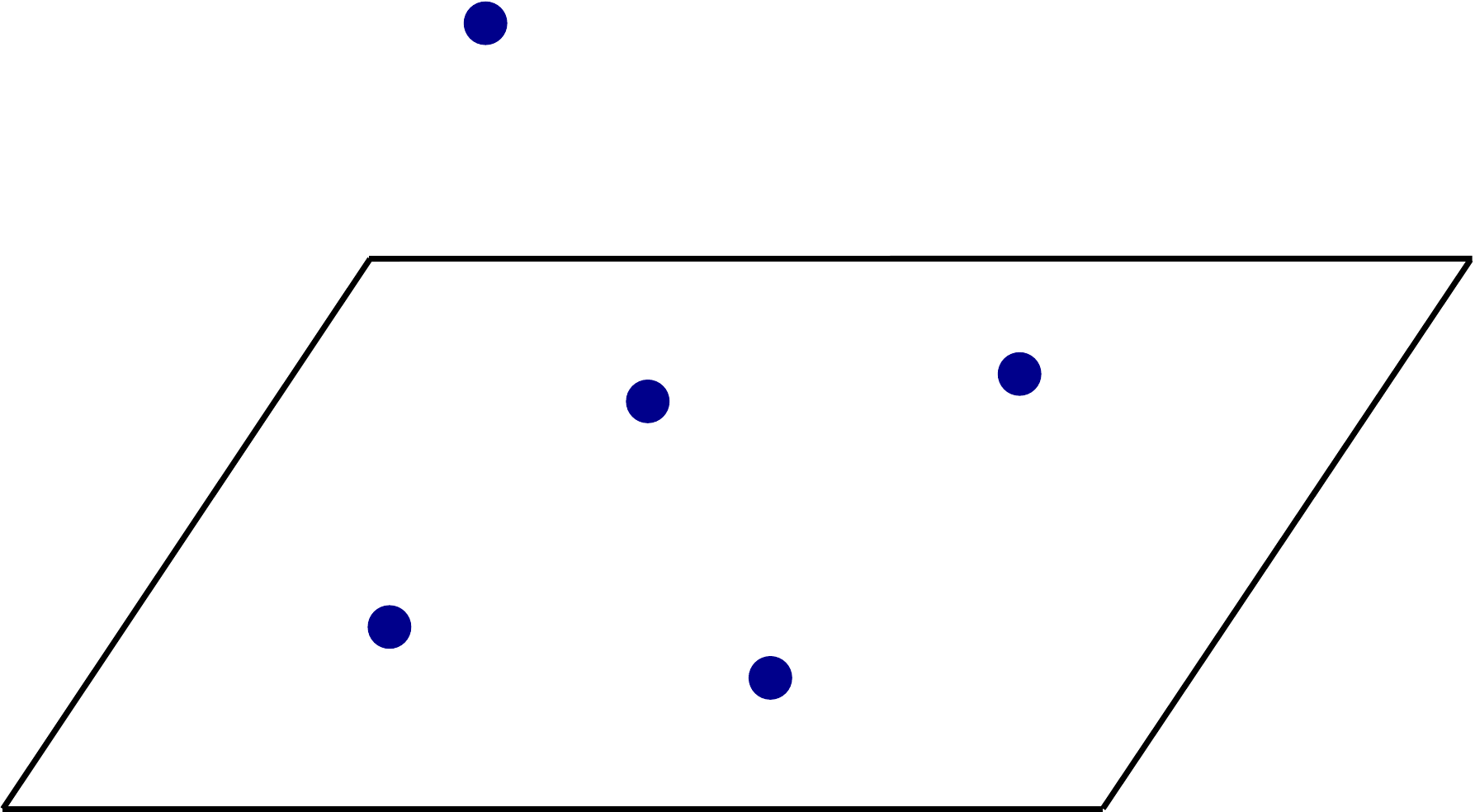}
					\begin{small}
						\put(12,4){$A_1$}
						\put(26,26){$A_2$}
						\put(56,4){$A_3$}
						\put(74,26){$A_4$}
						\put(16,54){$A_5$}
					\end{small} 
				\end{overpic}
			} &
			\subfigcapskip = 6pt
			\subfigure[][]{
			\label{figure:5space-c}
				\begin{overpic}[width=0.28\textwidth]{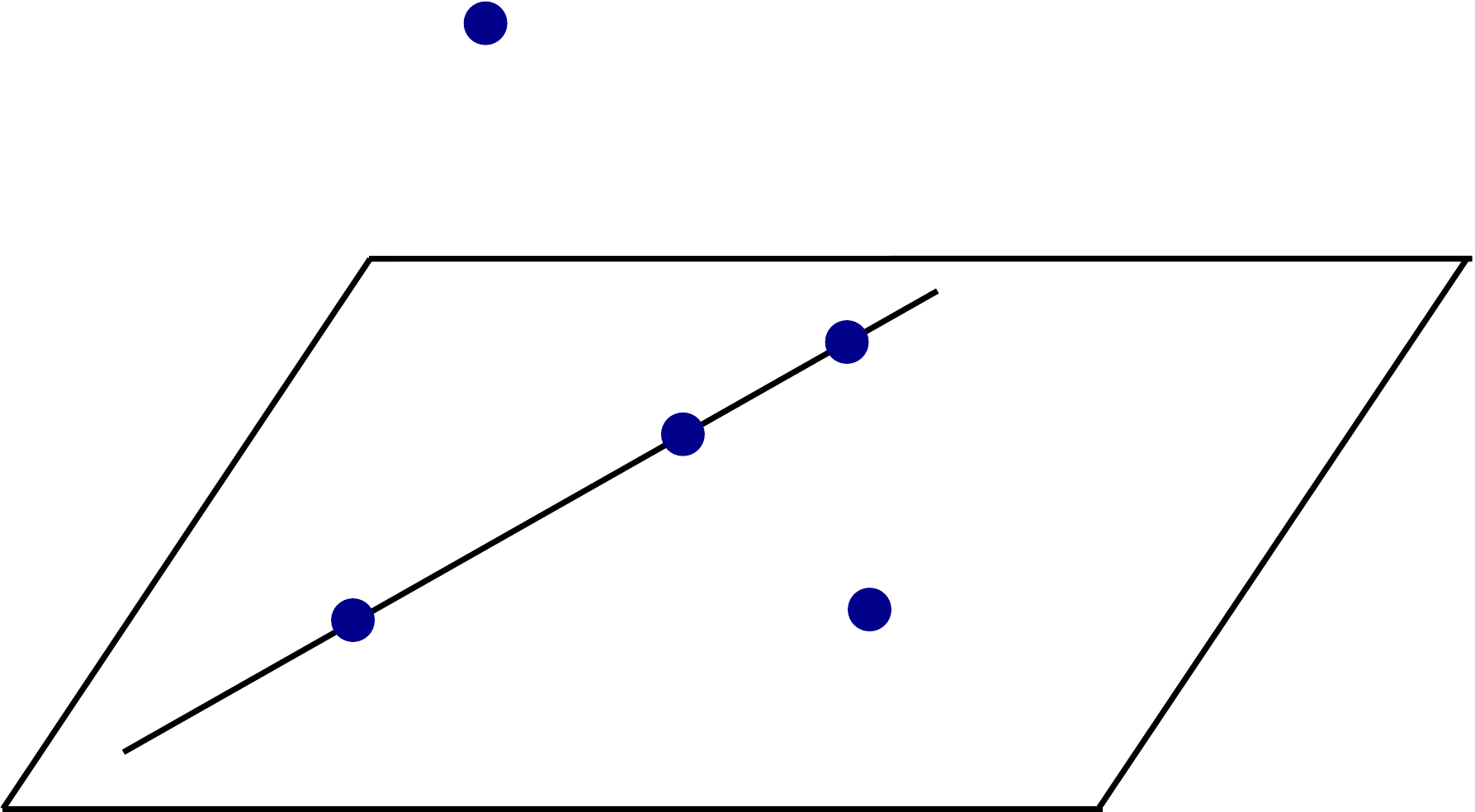}
					\begin{small}
						\put(24,4){$A_1$}
						\put(28,26){$A_2$}
						\put(60,26){$A_3$}
						\put(58,4){$A_4$}
						\put(16,54){$A_5$}
					\end{small} 
				\end{overpic}
			} 
		\end{tabular}
	\caption{Possible configurations of~$5$ non coplanar points in the space: 
\subref{figure:5space-a} no~$4$ points are coplanar, \subref{figure:5space-b} 
$4$ coplanar points, no~$3$ of them aligned, \subref{figure:5space-c} $3$ 
aligned points.}
	\label{figure:5space}
	\end{figure}
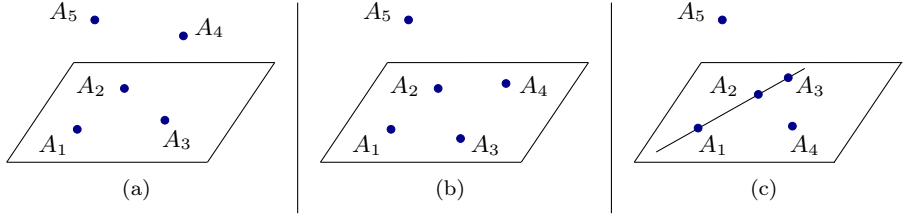
	\begin{description}
		\item[Case (a/b)]
			In this situation the components 
$\left( f_{\vec{A}} \right)_i$ of $f_{\vec{A}}$ do not have any common factor, 
hence either $f_{\vec{A}}$ is a birational map with image a degree~$10$ curve, 
or $f_{\vec{A}}$ is $2:1$ to a curve of degree~$5$. We prove that in the second 
case the points should be coplanar, so this cannot happen. If we suppose that
the map~$f_{\vec{A}}$ is $2:1$, we have the following:

\medskip
\noindent {\bf Claim.} It is possible to define a regular map $r_{\vec{A}}: C 
\longrightarrow C$ which respects the real structure on~$C$ and such that 
\[ r_{\vec{A}}^2 = \mathrm{id} \quad \mathrm{and} \quad 
 f_{\vec{A}}\left(r_{\vec{A}}(\varepsilon)\right) = f_{\vec{A}}(\varepsilon) \]
{\it Proof.} Suppose in fact that we are given a finite 
map $f: C \longrightarrow D$ where $C$ is a smooth curve and $f$ is generically 
$2:1$. If $\tilde{D}$ is the normalization of~$D$, we can lift~$f$ to a finite 
map $\tilde{f}: C \longrightarrow \tilde{D}$ which is also generically $2:1$. 
Then we define set-theoretically an involution $r: C \longrightarrow C$ in the 
following way: pick a point $P \in C$; in particular, $P$ is a prime divisor 
of~$C$, so we map it to $f^{*} \! \left( f_{*} (P) \right) - P$\footnote{Here
$f_{*}: \mathrm{Div}(C) \longrightarrow \mathrm{Div}(D)$ and $f^{*}:
\mathrm{Div}(D) \longrightarrow \mathrm{Div}(C)$ denote respectively the
\emph{pushforward} and the \emph{pullback} induced by $f$ between the
groups of divisors of the curves~$C$ and~$D$. For definitions and properties 
of these notions see, for example,~\cite{Hartshorne} Appendix A.}, which is 
also a prime divisor of~$C$, namely a point (we passed to the normalization in 
order to have good functorial properties of divisors; generically this map swaps 
the two elements in a fiber of~$f$). In order to prove that this map is 
regular, since the map~$f$ is generically $2:1$ we can suppose that locally it 
is given by the canonical injection $R \longrightarrow \faktor{R[x]}{(x^2 + b x 
+ c)} = S$, where $\mathrm{Spec}(R)$ is an open set in~$\tilde{D}$ and 
$\mathrm{Spec}(S)$ is an open set in~$C$. Hence $r$ is locally given by the 
homomorphism $S \longrightarrow S$ sending $x \mapsto -b -x$, which exchanges 
the two roots of $x^2 + b x + c$. In this way we see that $r$ is regular. 
Moreover, if~$C$ is a real variety and~$f$ is a real map, then also~$r$ is a 
real map. 

\medskip \noindent
If we think of~$C$ as the unit sphere~$S^2$, because of its properties
$r_{\vec{A}}$ has to be a rotation of~$S^2$ of $180^{\circ}$ along an 
axis, which also proves that $r_{\vec{A}}$ has two fixed points (the 
intersections of~$S^2$ with the axis of rotation). Recall the definition of
the lines~$L_{ij}$ in~$M_5$ (see Remark~\ref{remark:description_M_5}).
Then we get that 
\[ f_{\vec{A}}^{-1}(L_{ij}) \; = \; \left\{ \frac{A_i - A_j}{\left\| A_i -
A_j \right\|}, \frac{A_j - A_i}{\left\| A_i - A_j \right\|} \right\} \] 
On the other hand, if $\varepsilon \in f_{\vec{A}}^{-1}(L_{ij})$, then
also $r_{\vec{A}}(\varepsilon) \in f_{\vec{A}}^{-1}(L_{ij})$, so there are only
two options:
\begin{itemize}
	\item[i.]
		either $r_{\vec{A}}(\varepsilon) = -\varepsilon$, meaning that
$\varepsilon$ lies on a great circle of~$S^2$ (the one orthogonal to the axis
determined by $r_{\vec{A}}$) since $r_{\vec{A}}$ coincides with the antipodal
map only on this great circle;
	\item[ii.]
		or $r_{\vec{A}}(\varepsilon) = \varepsilon$, meaning that $\varepsilon$
is one of the two fixed points of $r_{\vec{A}}$.
\end{itemize}
			If possibility i. happens for every $L_{ij}$, this means that the
direction of all lines~$\overrightarrow{A_iA_j}$ lie on a great circle of~$S^2$, 
this implying that the points~$\{A_i\}$ are coplanar. If this were not
the case, since in our configuration no three~$A_i$ are collinear in
this case we have that possibility ii. can happen only for one line~$L_{ij}$.
Let us suppose that this line is~$L_{12}$: this would imply that the points
$A_2$, $A_3$, $A_4$ and $A_5$ are coplanar (in Case~(a) here we would have 
already reached a contradiction) and the line~$\overrightarrow{A_1A_2}$ is
orthogonal to the plane on which the other points lie. On the other hand, the
fact that all lines but $L_{12}$ fall on possibility i. implies that also $A_1$,
$A_2$, $A_3$ and $A_4$ are coplanar. Hence all points are coplanar. But this is
in contradiction with our assumption that the points~$\{A_i\}$ are not
coplanar. \\
We have shown that in this case the only situation which is left possible
is that the map~$f_{\vec{A}}$ is birational to a degree~$10$ curve.
		\item[Case (c)]
			Here we have that $H_{12}$, $H_{23}$ and $H_{13}$ are equal up to a scalar
factor, so the components of~$f_{\vec{A}}$ have one factor in common, which can
be removed. Thus four situations are possible: either $f_{\vec{A}}$ is
birational to a curve of degree~$8$, or it is $2:1$ to a curve of degree~$4$, or
it is $4:1$ to a conic, or it is $8:1$ to a line. Arguments similar to the ones
of Case~(a) in the proof of Lemma~\ref{lemma:5plane} rule out the last two
situations. In order to prove that the~$2:1$ situation is not possible, we
proceed as in Case~(a): here the image~$f_{\vec{A}}(C)$ does not meet all the
lines~$L_{ij}$, but from the configuration of the points $A_i$ it is ensured
that the curve intersects $L_{14}$, $L_{24}$, $L_{34}$, $L_{15}$, $L_{25}$,
$L_{35}$ and $L_{45}$, which is enough to prove that the points are coplanar.
\qedhere
	\end{description}
\end{proof}
\begin{remark}
\label{remark:constant_map}
	We notice that if all the~$5$ points of~$\vec{A}$ are aligned, then all 
the~$6$ components of the photographic map are proportional, hence $f_{\vec{A}}$ 
is a constant map.
\end{remark}
\begin{remark}
	Lichtblau stated the following conjecture (see Conjecture~2 
of~\cite{Lichtblau}; this was later proved by him in~\cite{Lichtblau2012}, 
Proposition~3 and Theorem~4): \emph{there can only exist an infinite number of 
cylinders of revolution passing through five distinct points in $\R^3$ if the 
points are located on two parallel lines}. Under the assumption that infinitely 
many circular cylinders are real we can use Lemma~\ref{lemma:5space} to give 
an alternative proof to the one of Lichtblau. In fact, if a $5$-tuple has 
such a property, then the image of its photographic map will have infinitely 
many real points (see Remark~\ref{remark:real_points_M_5}). Since $S^2$ does 
not have real points, it follows that the photographic map cannot be 
birational, hence the points have to be coplanar. From this it is well known 
that the points actually have to lie on two parallel lines. Therefore only the 
question if this condition is also necessary for the existence of infinitely 
many circular cylinders over~$\C$ passing through five real distinct points 
remains open.
\end{remark}
Now we state and prove the main result of this section.
\begin{theorem} 
\label{theorem:photo}
	Let $\vec{A}$ and $\vec{B}$ be two $5$-tuples of points in $\R^3$ such that
no~$4$ of them are collinear. Assume that $f_{\vec{A}}(C)$ and 
$f_{\vec{B}}(C)$ are equal as curves in~$M_5$.
	If $\vec{A}$ is coplanar, then $\vec{B}$ is also coplanar and affine
equivalent to~$\vec{A}$.
	If~$\vec{A}$ is not coplanar, then $\vec{B}$ is similar to~$\vec{A}$.
\end{theorem}
\begin{proof}
	Suppose that $\vec{A}$ is not coplanar. Then by Lemma~\ref{lemma:5space} we
know that $f_{\vec{A}}$ is birational to a curve of degree $10$ or $8$. From
Lemma~\ref{lemma:5plane} we have that $\vec{B}$ is also not coplanar, since
otherwise we would have a curve of different degree as the image 
of~$f_{\vec{B}}$. Thus $f_{\vec{B}}$ is birational, and by composing
$\tilde{f}_{\vec{A}}$ and $\tilde{f}_{\vec{B}}^{-1}$ we get an isomorphism
$\rho: \p^1_{\C} \stackrel{\cong}{\longrightarrow} \p^1_{\C}$ which respects the
real structure, since both $\tilde{f}_{\vec{A}}$ and $\tilde{f}_{\vec{B}}$ do
so. Thus $\rho$ is a rotation of~$S^2$. If we apply the rotation~$\rho$ 
to~$\vec{A}$ we obtain a vector of points~$\vec{A'}$ so that the diagram
	\[ \xymatrix{& D \\ \p^1_{\C} \ar[ru]^{\tilde{f}_{\vec{A'}}}
\ar[rr]_{\mathrm{id}} & & \p^1_{\C} \ar[lu]_{\tilde{f}_{\vec{B}}} }\]
	commutes, namely $\tilde{f}_{\vec{A}}$ and $\tilde{f}_{\vec{B}}$ coincide as 
maps. The goal now is to show that the direction~$\overrightarrow{A'_iA'_j}$ 
and $\overrightarrow{B_iB_j}$ coincide for every~$i$ and~$j$, this proving that 
$\vec{A'}$ and $\vec{B}$ are similar, from which we derive the thesis. Let us 
consider the situation when~$D$ has degree~$10$. Recall the definition of the
lines~$L_{ij}$ in~$M_5$ (see Remark~\ref{remark:description_M_5}). Analogously
as in the proof of Lemma~\ref{lemma:5space} we have that
	\[ f_{\vec{A'}}^{-1}(L_{ij}) \; = \;
\left\{ \frac{A'_i - A'_j}{\left\| A'_i - A'_j \right\|}, \frac{A'_j -
A'_i}{\left\| A'_i - A'_j \right\|} \right\} \]
	and similarly for~$f_{\vec{B}}^{-1}(L_{ij})$. Since the two maps
$\tilde{f}_{\vec{A'}}$ and $\tilde{f}_{\vec{B}}$ coincide our claim is proved.
In the situation when $D$ has degree~$8$ the argument is the same, but in this
case $D$ does not intersect all the lines~$L_{ij}$; however, knowing that
$f_{\vec{A'}}^{-1}(D \cap L_{ij})$ and $f_{\vec{B}}^{-1}(D \cap L_{ij})$ are
equal for $ij \in \{14, 24, 34, 15, 25, 35, 45 \}$ (see Case~(c) of 
Lemma~\ref{lemma:5space}) gives already enough information for proving that 
$\vec{A'}$ and $\vec{B}$ are similar.

	\smallskip
	Suppose that $\vec{A}$ is coplanar, then from Lemma~\ref{lemma:5plane} the 
map~$f_{\vec{A}}$ is $2:1$ to a curve of degree~$5$, $4$ or~$3$ (we avoid the 
conic case, since no~$4$ points are collinear by hypothesis; the reason for this
is clarified in Remark~\ref{remark:reconstruction_four_collinear}). Hence the
only possibility is that also $\vec{B}$ is coplanar, because otherwise from
Lemma~\ref{lemma:5space} we would get a curve of degree~$10$ or~$8$ as the image
of~$f_{\vec{B}}$. As in the proof of Lemma~\ref{lemma:5plane}, we know that both
$f_{\vec{A}}$ and $f_{\vec{B}}$ factor through a $2:1$ map to~$\p^1_{\C}$
followed by a birational map. By a change of coordinates we can suppose that 
this $2:1$ map is given by sending $(x:y:z) \mapsto (x:y)$. The picture of
the situation is:
	\[ \xymatrix{& & D & & \\ C \ar[r]_-{2:1} \ar[rru]^-{f_{\vec{A}}} & \p^1_{\C} \ar[ru]_-{\cong} & & \p^1_{\C} \ar[lu]^-{\cong} & C \ar[l]^-{2:1} \ar[llu]_-{f_{\vec{B}}} } \]
	Thus we get an isomorphism $\p^1_{\C} \stackrel{\cong}{\longrightarrow} \p^1_{\C}$ which makes the previous diagram commute. If $M$ is the invertible $2 \times 2$ matrix representing it, and we denote by $A'$ the vector of points obtained by applying the affinity associated to $M$ to $\vec{A}$, then the following diagram commutes:
	\[ \xymatrix@R=.4pc{& D & \\ \\ \\ C \ar[r] \ar[ruuu]^-{f_{\vec{A'}}} & \p^1_{\C} \ar[uuu] & C \ar[l] \ar[luuu]_-{f_{\vec{B}}} \\ \scriptstyle (x:y:z) \ar@{|->}[r] & \scriptstyle (x:y) & \scriptstyle (x:y:z) \ar@{|->}[l]} \]
	In this way we reached the point where $f_{\vec{A'}}$ and $f_{\vec{B}}$ are equal as maps, thus we can proceed as in the non planar case, proving that $A'$ and $B$ are similar, so $\vec{A}$ and $\vec{B}$ are affine equivalent. 
\end{proof}
\begin{remark}
	We can describe an algorithm which takes as an input the image of the
photographic map of a vector of points~$\vec{A}$ satisfying the conditions of
Theorem~\ref{theorem:photo} and gives back a vector of points~$\vec{C}$ which is
similar to~$\vec{A}$. In Algorithm~1 we describe
the procedure in the case of non planar points, when the degree of~$D$ is~$10$.
This is the easiest situation, because we have information about all the
directions of the lines passing through the points of~$\vec{A}$.
	\begin{algorithm}
	\caption{Non planar point reconstruction}
		\begin{algorithmic}[1]
			\Require $D \subseteq M_5$, a degree~$10$ curve such that 
$f_{\vec{A}}(S^2) = D$.
			\Ensure $\vec{C}$ such that it is similar to $\vec{A}$.
			\Statex
			\State {\bfseries Parametrize} $D$ via $\varphi$ respecting the real
structure of $D$.
			\State {\bfseries Compute} $\big\{ \varepsilon_{ij}, -\varepsilon_{ij}
\big\} = \varphi^{-1}(L_{ij})$ for all $i,j$.
			\State {\bfseries Set} $C_1 = (0,0,0)$.
			\State {\bfseries Pick} $C_2$ arbitrary on the line $\big\{ C_1 + t \tth \varepsilon_{12} \, : \; t \in \R \big\}$.
			\State {\bfseries Construct} $C_3$ as the intersection of the lines $\big\{ C_1 + t \tth \varepsilon_{13} \big\}$ and $\big\{ C_2 + t \tth \varepsilon_{23} \big\}$.
			\State {\bfseries Construct} $C_4$ using $\varepsilon_{24}$ and
$\varepsilon_{34}$ as in Step~5.
			\State {\bfseries Construct} $C_5$ using $\varepsilon_{35}$ and
$\varepsilon_{45}$ as in Step~5.
			\State \Return $\vec{C} = (C_1, \ldots, C_5)$.
		\end{algorithmic}
	\end{algorithm}

	We notice that we can always perform Steps~$5$, $6$ and~$7$, namely, the 
involved lines always intersect. This is ensured by the fact that we 
start from an existing configuration of points.

	When the curve~$D$ has degree~$8$, $5$, $4$ or~$3$ the algorithm is almost 
the same, we just have to take into account that $D$ will not intersect all the 
lines $L_{ij}$: the ones which are disjoint from the image of~$f_{\vec{A}}$ 
reveal which points in~$\vec{C}$ will be collinear, and the others can be 
used to identify the whole configuration. 
\end{remark}
\begin{remark}
\label{remark:reconstruction_four_collinear}
	We notice that we have to avoid the case when~$4$ points are collinear
(namely when the degree of the image of the photographic map is~$2$), because in
that case it is not possible to reconstruct the direction
$\overrightarrow{A_1A_4}$. In fact, $f_{\vec{A}}(C) \cap L_{14} = \emptyset$
since projecting in that direction would give a configuration where four points
coincide, which is not allowed in~$M_5$. In this case one can show that the
images of two photographic maps $f_{\vec{A}}(C)$ and $f_{\vec{B}}(C)$ are equal
if and only if the cross ratios of the two $4$-tuples of collinear points are
equal. On the other hand, also when we only have three aligned points the image
of the photographic map does not intersect the line~$L_{13}$, but in this case
we can reconstruct the whole configuration regardless the knowledge 
of~$\overrightarrow{A_1A_3}$, since we can use $\overrightarrow{A_1A_4}$ and
$\overrightarrow{A_1A_5}$ to determine $A_1$ starting from $A_4$ and $A_5$, and
do the same for $A_2$ and $A_3$ --- this procedure cannot be applied to the
previous configuration. The two situations are described in 
Figure~\ref{figure:reconstruction}.
	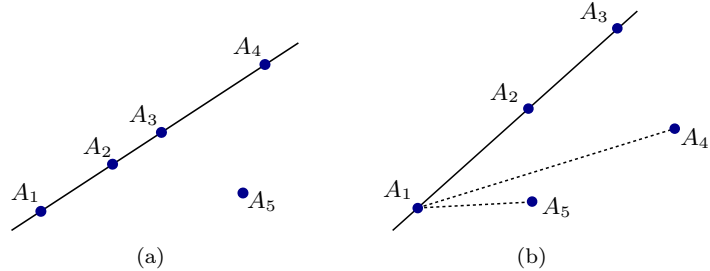
\begin{figure}[ht]
		\begin{tabular}{cc}
			\subfigcapskip = 6pt
			\subfigure[][]{
			\label{figure:four_collinear}
				\begin{overpic}[width=0.3\textwidth]{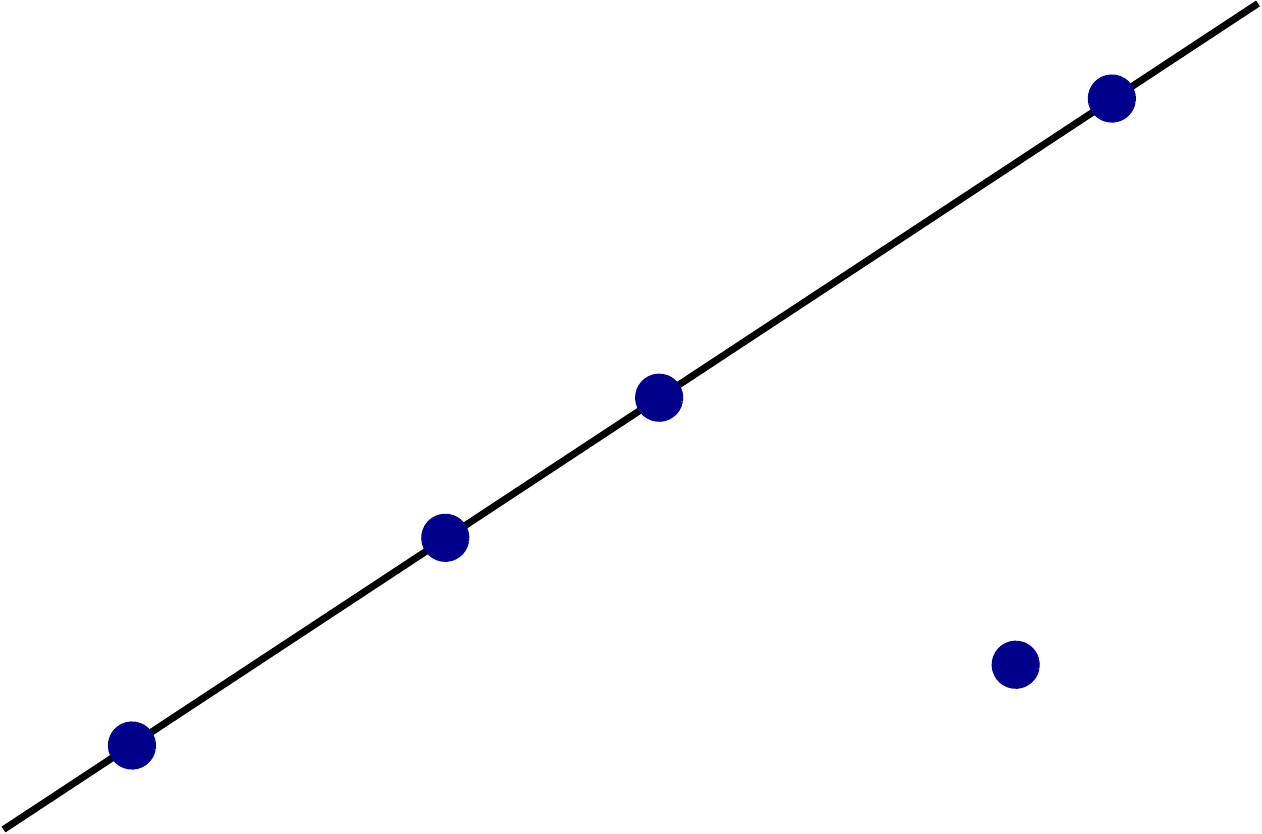}
					\begin{small}
						\put(0,11){$A_1$}
						\put(25,27){$A_2$}
						\put(41,38){$A_3$}
						\put(77,62){$A_4$}
						\put(82,8){$A_5$}
					\end{small} 
				\end{overpic}
			} \hspace{0.2cm} & \hspace{0.2cm}
			\subfigcapskip = 6pt
			\subfigure[][]{
			\label{figure:three_collinear}
				\begin{overpic}[width=0.3\textwidth]{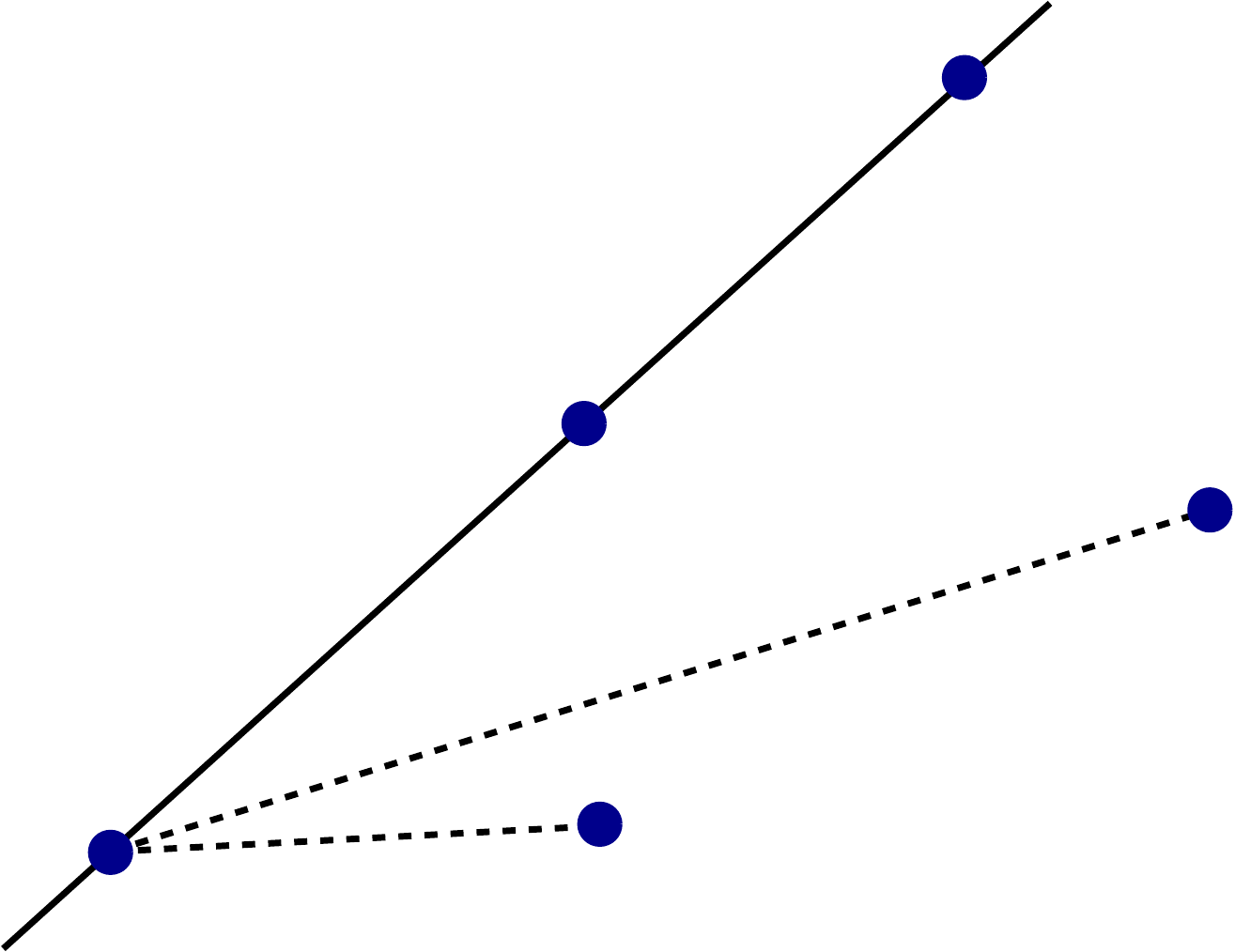}
					\begin{small}
						\put(-3,12){$A_1$}
						\put(34,46){$A_2$}
						\put(65,74){$A_3$}
						\put(100,30){$A_4$}
						\put(52,6){$A_5$}
					\end{small} 
				\end{overpic}
			}
		\end{tabular}
	\caption{In the case of \subref{figure:four_collinear} four collinear points,
the reconstruction algorithm does not work, since it is not possible to recover
the direction of the line on which the four points lie. Instead, if we only
allow \subref{figure:three_collinear} three collinear points, then the algorithm
succeeds since we can reconstruct the aligned points using the other ones.}
	\label{figure:reconstruction}
	\end{figure}
\end{remark}

Eventually, it is possible to extend the consequences of 
Theorem~\ref{theorem:photo} to tuples of~$n$ points when $n > 5$. In order to do
this, starting from such an $n$-tuple $\vec{A}$ one can define a photographic
map~$f_{\vec{A}}: C \longrightarrow M_n$, where $M_n$ is the moduli space of~$n$
points in $\p^1_{\C}$, in the same way as we did in this paper. Then for every
sub-tuple of~$5$ elements of $\vec{A}$, say $(A_1, \ldots, A_5)$, one has a
commutative diagram:
\begin{equation}
\label{equation:diagram_cor}
	\xymatrix{C \ar[rr]^{f_{\vec{A}}} \ar[rd]_-{f_{(A_1, \ldots,
A_5)}} & & M_n \ar@{-->}[ld]^-{\delta_{(1, \ldots, 5)}} \\ & M_5}
\end{equation}
where $\delta_{(1, \ldots, 5)}$ associates the equivalence class of the
$n$-tuple $(m_1, \ldots, m_n)$ to the equivalence class of the $5$-tuple
$(m_1, \ldots, m_5)$ (this is a rational map). 

\begin{cor}
\label{cor:photo}
	Theorem~\ref{theorem:photo} holds true also if we take $\vec{A}$ and
$\vec{B}$ to be two $n$-tuples of points in~$\R^3$ where no $n-1$ points are
collinear, provided that~$n \geq 5$.
\end{cor}
\begin{proof}
	We prove the statement by reducing to the $n = 5$ case and applying 
Theorem~\ref{theorem:photo}. Suppose that~$\vec{A}$ is not coplanar; we want to
prove that $\vec{A}$ and $\vec{B}$ are similar. After possibly
relabeling the points, we can suppose that $A_1, \ldots, A_4$ are not coplanar.
By hypothesis we have that $f_{\vec{A}}(C) = f_{\vec{B}}(C)$, so from 
Diagram~\ref{equation:diagram_cor} we can infer that for every $k \geq 5$ we 
have $f_{(A_1, \ldots, A_4, A_k)}(C) = f_{(B_1, \ldots, B_4, B_k)}(C)$. Thus by
Theorem~\ref{theorem:photo} we get that for every $k \geq 5$, the two
$5$-tuples $(A_1, \ldots, A_4, A_k)$ and $(B_1, \ldots, B_4, B_k)$ are not
coplanar and similar. Now, since in this case there exists a unique similarity
sending $(A_1, \ldots, A_4)$ to $(B_1, \ldots, B_4)$, from what we said we
have that the same similarity sends $A_k$ to $B_k$ for all $k \geq 5$. Hence
$\vec{A}$ and $\vec{B}$ are similar. \\
If $\vec{A}$ is coplanar, then from the commutativity of 
Diagram~\ref{equation:diagram_cor} and by Theorem~\ref{theorem:photo} we obtain 
that also $\vec{B}$ is coplanar. Now we can proceed as before to get the thesis,
but here in order to be able to use Theorem~\ref{theorem:photo} we have to make
sure that we can choose $A_1, \ldots, A_4$ so that for every $k \geq 5$ there
are no~$4$ collinear points among $A_1, \ldots, A_4, A_k$. This is ensured by
the hypothesis that no~$n-1$ among the~$\{ A_i \}$ are collinear, since the
latter is the only case when this choice cannot be made. Hence we can conclude
as before, since an affinity is completely determined by the image of~$3$ non
collinear points. 
\end{proof}

\section{A necessary condition for pentapods with mobility 2}
\label{photogrammetry:pentapods}

We can finally apply the theory we developed so far to get necessary conditions
for mobility of pentapods. The geometry of this kind of mechanical manipulators
is defined by the coordinates of the~$5$ platform anchor points $p_1, \ldots,
p_5 \in \R^3$ and of the~$5$ base anchor points $P_1, \ldots, P_5 \in \R^3$ in
one of their possible configurations. All pairs of points $(p_i, P_i)$ are
connected by a rigid body, called \emph{leg}, so that for all possible
configurations the distance $d_i = \left\| p_i - P_i \right\|$ is preserved.
The dimension of the set of possible configurations of a pentapod is called its
\emph{mobility} (for a formal definition of this concept, see 
\cite{GalletNawratilSchicho}, Section 3, Definition 3.2).

In \cite{GalletNawratilSchicho} we proved the following
conditions for $n$-pods (replace~$5$ by~$n$ in the previous paragraph):

\begin{theorem}
\label{theorem:mobility_two}
	Let $\Pi$ be an $n$-pod with mobility~$2$ or higher. Then one of the
following holds:
	\begin{itemize}
		\item[(a)]
			there are infinitely many pairs $(L,R)$ of elements of $S^2$ such that the
points $\pi_L(p_1), \ldots, \pi_L(p_n)$ and $\pi_R(P_1), \ldots, \pi_R(P_n)$
differ by an inversion or a similarity;
		\item[(b)]
			there exists $m \leq n$ such that $p_1, \ldots, p_m$ are collinear and
$P_{m+1} = \ldots = P_n$, up to permutation of indices and interchange between
base and platform;
		\item[(c)]
			there exists $m \leq n$ with $1 < m < n-1$ such that $p_1, \ldots, p_m$
lie on a line $g \subseteq \R^3$ and $p_{m+1}, \ldots, p_n$ are located on a
line $g' \subseteq \R^3$ parallel to~$g$; moreover $P_1, \ldots, P_m$ lie on a
line $G \subseteq \R^3$ and $P_{m+1}, \ldots, P_n$ are located on a line $G'
\subseteq \R^3$ parallel to~$G$, up to permutation of indices.
	\end{itemize}
\end{theorem}

For $n = 5$ we can use our M\"obius Photogrammetry technique to reformulate
condition~(a) above in a more geometric fashion. 

\begin{theorem} 
\label{theorem:5pod}
		Let $\Pi$ a pentapod with mobility~$2$ or higher. Then one of the	following
conditions holds:
	\begin{itemize}
		\item[(a)] the platform and the base are similar;
		\item[(b)] the platform and the base are planar and affine equivalent;
		\item[(c)] there exists $m \leq 5$ such that $p_1, \ldots, p_m$ are
collinear and $P_{m+1}, \ldots, P_5$ coincide, up to permutation of indices and
interchange of platform and base;
		\item[(d)]
			the points $p_1, p_2, p_3$ lie on a line $g \subseteq \R^3$ and $p_{4},
p_5$ lie on a line $g' \subseteq \R^3$ parallel to~$g$, and $P_1, P_2, P_3$
lie on a line $G \subseteq \R^3$ and $P_{4}, P_5$ lie on a line $G' \subseteq
\R^3$ parallel to~$G$, up to permutation of indices.
	\end{itemize}
\end{theorem}

\begin{proof}
Since $\Pi$ has mobility at least~$2$, then by 
Theorem~\ref{theorem:mobility_two} either we are in cases~(c) or~(d), or there
are infinitely many pairs $(L,R)$ of elements of $S^2$ such that the
points $\pi_L(p_1), \ldots, \pi_L(p_5)$ and $\pi_R(P_1), \ldots, \pi_R(P_5)$
differ by an inversion or a similarity. Let us consider then this last
case. Since we can suppose that no~$4$ point of the base or platform are aligned
(otherwise we are in case~(c) or~(d)), we have in particular that the
photographic maps $f_{\vec{P}}$ and $f_{\vec{p}}$ of base and platform points of
$\Pi$ are not constant. Hence, if we re-interpret the assumption in the language
we developed in this paper, we have that the images~$f_{\vec{P}}(C)$ 
and~$f_{\vec{p}}(C)$ have infinitely points in common. Since both are 
irreducible algebraic curves, they must coincide, and we get~(a) or~(b) by
Theorem~\ref{theorem:photo}.
\end{proof}

\begin{remark}
	For quadropods the analogous statement of Theorem~\ref{theorem:5pod}
does not hold. In fact, all quadropods have mobility at least~$2$, but the
general quadropod does not fulfill any of the conditions (a)--(d) of the
theorem. For tripods the statement is trivially true, since conditions~(b)
and~(c) are always fulfilled. 

	For $n$-pods with $n > 5$ one can prove a statement analogous to Theorem~
\ref{theorem:5pod} by using Corollary~\ref{cor:photo}.
\end{remark}

Based on Theorem~\ref{theorem:5pod} a complete classification of pentapods with
mobility~$2$ was given in~\cite{NawratilSchicho1} and~\cite{NawratilSchicho2}.

\section*{Acknowledgments}

The first and third author's research is supported by the Austrian Science Fund
(FWF): W1214-N15/DK9 and P26607 - ``Algebraic Methods in Kinematics: Motion
Factorisation and Bond Theory''. The second author's research is funded by the
Austrian Science Fund (FWF): P24927-N25 - ``Stewart Gough platforms with
self-motions''.

\end{document}